\documentclass[11pt,reqno]{amsart}

\usepackage{amsmath}
\usepackage{amssymb}
\usepackage{amsthm}

\newtheorem{theorem}{Theorem}[section]
\newtheorem{prop}[theorem]{Proposition}
\newtheorem{lem}[theorem]{Lemma}

\theoremstyle{definition}
\newtheorem{defn}[theorem]{Definition}
\theoremstyle{remark}
\newtheorem*{rem}{Remark}
\numberwithin{equation}{section}

\newcommand{\Gg}{\mathfrak{g}}    
\newcommand{\Gh}{\mathfrak{h}}
\newcommand{\Gq}{\mathfrak{q}}

\newcommand{\Gz}{\mathfrak{z}}

\begin{document}

\title[Quantum group constructed from a skew-symmetric matrix]
{A $(2n+1)$-dimensional quantum group constructed from a skew-symmetric matrix}

\author{Byung-Jay Kahng}
\date{}
\address{Department of Mathematics and Statistics\\ Canisius College\\
Buffalo, NY 14208}
\email{kahngb@canisius.edu}

\begin{abstract}
Beginning with a skew-symmetric matrix, we define a certain Poisson--Lie group. 
Its Poisson bracket can be viewed as a cocycle perturbation of the linear (or 
``Lie--Poisson'') Poisson bracket.  By analyzing this Poisson structure, we gather 
enough information to construct a $C^*$-algebraic locally compact quantum group, 
via the ``cocycle bicrossed product construction'' method.  The quantum group thus 
obtained is shown to be a deformation quantization of the Poisson--Lie group, in the 
sense of Rieffel.
\end{abstract}
\maketitle

\section{Introduction}

It is generally understood that quantum groups are obtained by ``quantizing'' 
ordinary groups.  On the other hand, it is not always clear what we mean by this 
statement.  Typically, in the often used ``generators and relations method'' 
of constructing quantum groups, there exists a certain deformation parameter $q$ 
such that when $q=1$, the quantum group degenerates to the universal enveloping 
algebra of an ordinary group or the function algebra of an ordinary group.  See 
\cite{RTF}, \cite{Wr1}, and other examples.  While this is nice, the method of 
generators and relations is at best an indirect method, meaning that the correspondence 
information about how the pointwise product on the function algebra is deformed 
to an operator product is usually not apparent.

There are also some technical issues when working with the $q$-relations among the 
generators.  It is less of a problem in the case of a purely algebraic setting of 
quantized universal enveloping (QUE) algebras or that of compact quantum groups. 
However, when one wishes to construct a non-compact quantum group, the generators 
(coordinate functions of the group) tend to be unbounded, so things are more 
complicated.  There are ways to handle the difficulties (See \cite{Wr4}, where 
Woronowicz works with the notion of unbounded operators ``affiliated'' with 
$C^*$-algebras), but in general, it is usually better to look for some other methods 
of construction.

One useful approach not relying on the generators is the method of deformation 
quantization.  Here, the aim is to deform the (commutative) algebra of functions 
on a Poisson manifold, in the direction of the Poisson bracket.  See \cite{BFFL}, 
\cite{Vy}.  In the $C^*$-algebra framework, the corresponding notion is the 
``strict deformation quantization'' by Rieffel \cite{Rf1}, or its more generalized 
versions developed later by other authors.  To obtain a quantum group, one would 
begin with a suitable Poisson--Lie group $G$ (a Lie group equipped with a compatible 
Poisson bracket) and perform the deformation quantization on the function space 
$C_0(G)$---for both its algebra and coalgebra structures.

Some of the non-compact quantum groups obtained by deformation quantization are 
\cite{Rf5}, \cite{SZ}, \cite{Rf7}, \cite{Zk2}, \cite{BJKp2}.  In these examples, 
there exists a very close relationship between a quantum group and its Poisson--Lie 
group counterpart.  Indeed, the information at the level of Poisson--Lie groups 
or Lie bialgebras plays a key role in the construction of the quantum groups.

The interplay between the Poisson data and the quantum group can go further. 
For instance, as for the example constructed by the author \cite{BJKp2}, the 
information at the classical (Poisson) level was useful not only in the construction 
of the quantum group but also in studying its representation theory, in relation to 
the dressing orbits.  See \cite{BJKppdress}.

Despite many advantages, however, jumping from a Poisson--Lie group to the 
$C^*$-algebraic quantum group level is not always easy.  Deformation quantization 
only provides the ``spatial'' quantization.  Even with the guides suggested by the 
Poisson data, the actual construction of the structure maps like comultiplication, 
antipode, or Haar weight requires various specialized techniques.  Often, a method 
that works for some examples may not work for others.

Considering the drawbacks to the geometric approach above, we proposed in \cite{BJKjgp} 
to enhance the ``geometric'' (deformation quantization) approach by combining it with 
a more ``algebraic'' framework of {\em cocycle bicrossed products\/}  \cite{Mj1}, \cite{VV}.

The bicrossed product method is relatively simple, but sufficiently general to include 
many interesting examples.  Historically, it goes back to the group extension problems 
in the Kac algebra setting.  For a comprehensive treatment on constructing quantum 
groups using this framework, see \cite{VV}.  However, as is the case for any general 
method, having the framework is not enough to construct actual and specific examples: 
For this method to work, one needs to have a specific ``matched pair'' of groups (or 
quantum groups), together with a compatible cocycle.  

Our proposal, as given in \cite{BJKjgp}, is to begin first with a Poisson--Lie group 
and analyze its Poisson structure.  The Poisson data will help us obtain a suitable 
matched pair and a compatible cocycle.  If, in particular, the Poisson bracket is of 
the ``cocycle perturbation of the linear Poisson bracket'' type, in the sense of 
\cite{BJKp1}, then the deformation process can be made more precise.  Finally, using 
the matched pair and the cocycle data, we will perform the cocycle bicrossed product 
construction.

Quantum groups obtained in this way tend to have (twisted) crossed products as their 
underlying $C^*$-algebras.  And therefore, this program is usually best for constructing 
solvable-type quantum groups.  It is because crossed products often model quantized spaces 
(For instance, the ``Weyl algebra'', $C_0(\mathbb{R}^n)\rtimes_{\tau}\mathbb {R}^n$ with 
$\tau$ being the translation, is the quantized phase space \cite{Fophasespace}.).  But with 
some adjustments, the method could be adopted to construct other types of quantum groups. 

The previous paper \cite{BJKjgp} gave examples of some Poisson--Lie groups and implicitly 
indicated how one may be able to carry out the program, but it never contained any 
detailed construction.  Case~(1) of \cite{BJKjgp} is related with the examples from 
\cite{Rf5}, \cite{SZ}, \cite{VD}, while Case~(2) was studied in \cite{BJKp2}.  However, 
these earlier papers did not exactly take the approach that we are proposing here.

The reason behind writing this paper is that in addition to giving an example of a quantum 
group, we wanted to expand on our work in the previous paper \cite{BJKjgp} by providing 
a careful description of our construction method, taking the Case~(3) in that paper as 
a model.  The author was initially content with the brief description as given in the 
previous paper, but while visiting Leuven during November 2008, he was suggested by 
Professor Alfons Van Daele that it would be beneficial to give a fuller description of 
the example and the method.  This is done here.  We expect that our program can be used 
to construct other new examples in a similar way.  Moreover, since we would have a close, 
built-in connection between the Poisson--Lie group and the quantum group, we will be able 
to take advantage of the geometric data in further studying the quantum group and applications 
(for instance, dressing orbits on Poisson--Lie groups are closely related with the quantum 
group representations).

The paper is organized as follows.  In Section~2, using a given skew-symmetric matrix $J$, 
we define the Poisson--Lie group $G$ that we wish to quantize.  Its Poisson bracket is 
non-linear, but can be regarded as a ``cocycle perturbation'' of the linear (Lie-Poisson 
type) Poisson bracket.
The deformation quantization of the Poisson--Lie group $\bigl(G,\{\ ,\ \}\bigr)$ is 
carried out in Section~3.  The Poisson data helps us to define a certain multiplicative 
unitary operator (in the sense of Baaj and Skandalis \cite{BS}), and it enables us to 
define the $C^*$-bialgebra $(S,\Delta)$.  It is shown here that $(S,\Delta)$ is a strict 
deformation quantization (in the sense of Rieffel \cite{Rf1}, \cite{Rf4}) of the 
Poisson--Lie group $G$.  In Section~4, we realize that the construction we carry out 
in Section~3 is in fact a case of the cocycle bicrossed product construction, in the 
sense of \cite{VV}.  The result is that the $C^*$-bialgebra is indeed a locally compact 
quantum group.  To tie the loose ends, brief descriptions are given on the antipode map 
and the Haar weight on our quantum group $(S,\Delta)$.

\section{The Poisson--Lie group $G$}

\subsection{The group}

Let $n$ be an integer such that $n\ge2$, and let $J=(J_{ik})_{1\le i,k\le n}$ 
be an $n\times n$ skew-symmetric matrix.  So $J_{ki}=-J_{ik}$, for $1\le i,k\le n$. 
Then consider the $(2n+1)$-dimensional Lie algebra $\Gg$, spanned by the basis 
elements $\mathbf{p}_i,\mathbf{q}_i \ (i=1,\dots,n),\mathbf{r}$, satisfying the 
following relations: 
$$
[\mathbf{p}_i,\mathbf{p}_j]=0,\quad[\mathbf{q}_i,\mathbf{q}_j]=0,\quad
[\mathbf{p}_i,\mathbf{q}_j]=0,\quad[\mathbf{p}_i,\mathbf{r}]=
\sum_{k=1}^nJ_{ik}\mathbf{q}_k,\quad[\mathbf{q}_i,\mathbf{r}]=0,
$$
for $i,j=1,\dots,n$.  Since $J=(J_{ik})$ is skew, it is clear that $[\ ,\ ]$ is 
a valid Lie bracket.  Observe also that the $\mathbf{q}_j$ are central and that 
$\Gg$ is a two-step nilpotent Lie algebra.

It is not difficult to describe the corresponding Lie group.  The group $G$ has 
$\mathbb{R}^{2n+1}$ as its underlying space, and the multiplication on it is 
defined by
\begin{equation}\label{(groupmult)}
(p,q,r)(p',q',r')=\left(p+p',q+q'+r'\sum_{i,k=1}^n J_{ik}p_i
\mathbf{q}_k,r+r'\right).
\end{equation}
Here, $p,q,p',q'\in\mathbb{R}^n$ and $r,r'\in\mathbb{R}$.  For convenience, 
we are regarding $p\in\mathbb{R}^n$ as $p=p_1\mathbf{p}_1+p_2\mathbf{p}_2+\dots
+p_n\mathbf{p}_n$, and similarly for the other variables.  In other words, 
the multiplication law in \eqref{(groupmult)} could be also written as:
\begin{align}
(p,q,r)(p',q',r')&=\left(p_1+p'_1,\cdots,p_n+p'_n;\right.  \notag \\
&\left.\quad q_1+q'_1+r'\sum_{i=1}^n J_{i1}p_i,
\cdots,q_n+q'_n+r'\sum_{i=1}^n J_{in}p_i;r+r'\right).
\notag
\end{align}
The identity element is $e=(0,0,0)$, while the inverse element for $(p,q,r)\in G$ is:
$$
(p,q,r)^{-1}=\left(-p,-q+r\sum_{i,k=1}^n J_{ik}p_i\mathbf{q}_k,-r\right).
$$

The group $G$ is a (connected and simply connected) exponential solvable Lie group 
corresponding to $\Gg$.  We can identify $G\cong\Gg$ as vector spaces.  Note that 
an ordinary Lebesgue measure becomes a Haar measure for $G$.

In the following section, we will further show that $G$ is equipped with a compatible 
Poisson bracket, making it a Poisson--Lie group.  Our aim in this paper is to construct 
a locally compact quantum group that can be considered as a ``quantized $C_0(G)$''.

\subsection{Non-linear Poisson structure on $G$}

By general theory on Poisson--Lie groups (see \cite{CP}), any compatible Poisson 
structure on $G$ canonically determines a dual Poisson--Lie group, and vice versa. 
In fact, in our case, it is in some sense more convenient to consider first its 
dual counterpart $H=G^*$, which is shown to be a Poisson--Lie group.  We can then 
regard $G$ as the dual Poisson--Lie group of $H$.  The following discussion was 
first reported in our previous paper: see Section 1, Case~(3) of \cite{BJKjgp}.

\begin{defn}(Heisenberg Lie group)
Let $H$ be the $(2n+1)$-dimensional {\em Heisenberg Lie group\/}.  Its underlying 
space is $\mathbb{R}^{2n+1}$ and the multiplication on it is given by
$$
(x,y,z)(x',y',z')=\bigl(x+x',y+y',z+z'+\beta(x,y')\bigr),
$$
for $x,x',y,y'\in\mathbb{R}^n$ and $z,z'\in\mathbb{R}$.  Here, $\beta(\ ,\ )$ 
denotes the ordinary inner product.  So $\beta(x,y)=x\cdot y$, for $x,y\in\mathbb{R}^n$.

Its Lie algebra counterpart is the Heisenberg Lie algebra $\Gh$.  It is generated 
by the basis elements $\mathbf{x}_i,\mathbf{y}_i (i=1,\dots,n),\mathbf{z}$, with 
the following relations:
$$
[\mathbf{x}_i,\mathbf{y}_j]=\delta_{ij}\mathbf{z}, 
\quad [\mathbf{x}_i,\mathbf{x}_j]=[\mathbf{y}_i,\mathbf{y}_j]=0,
\quad [\mathbf{z},\mathbf{x}_i]=[\mathbf{z},\mathbf{y}_i]=0.
$$
\end{defn}

\begin{rem}
For convenience, we will identify $H\cong\Gh$ as vector spaces.  This is possible since 
$H$ is an exponential solvable Lie group (it is actually nilpotent).  And, we choose 
a Lebesgue measure on $H\cong\Gh$, which is in fact a Haar measure for $H$.  As in 
Section 2.1, we will understand that $x=x_1\mathbf{x}_1+\dots+x_n\mathbf{x}_n$, and similarly 
for the other variables.
\end{rem}

To describe the Poisson structure on $H$, it is equivalent to specify a ``Lie bialgebra'' 
structure $(\Gh,\delta)$.  See \cite{LW}, \cite{CP}, for the general theory on Poisson--Lie 
groups and Lie bialgebras.  In our case, the cobracket $\delta:\Gh\to\Gh\wedge\Gh$ is obtained 
from a certain {\em classical $r$-matrix\/}.  Details are given in the following proposition. 
See also Section 5 of \cite{BJKjgp}.

\begin{prop}
Let $r\in{\Gh}\otimes{\Gh}$ be defined by $r=\sum_{i,k=1}^n J_{ik}\mathbf{x}_k\otimes\mathbf{x}_i$. 
It is a skew solution of the ``classical Yang--Baxter equation'' (CYBE):
$$
[r^{12},r^{13}]+[r^{12},r^{23}]+[r^{13},r^{23}]=0.
$$
Therefore, it determines a ``triangular'' Lie bialgebra structure, $\delta:\Gh\to\Gh\wedge\Gh$, 
by $\delta(X)=\operatorname{ad}_X(r)$, $X\in\Gh$.  To be specific, we have:
$$
\delta(\mathbf{x}_k)=0, \quad
\delta(\mathbf{y}_k)
=\sum_{i=1}^n J_{ik}(\mathbf{x}_i\otimes\mathbf{z}-\mathbf{z}\otimes\mathbf{x}_i)
=\sum_{i=1}^n J_{ik}\mathbf{x}_i\wedge\mathbf{z}, \quad
\delta(\mathbf{z})=0,  
$$
for $k=1,\dots,n$.
\end{prop}

\begin{proof}
Since $\operatorname{span}(\mathbf{x}_i:i=1,2,\dots,n)$ is an abelian subalgebra of $\Gh$, 
the element $r$ trivially satisfies the CYBE.  It is also skew (i.\,e. $r^{12}+r^{21}=0$), 
because $J$ is a skew-symmetric matrix.  This means that $r$ is a triangular classical 
$r$-matrix. 

By general theory (see, for instance, \cite{Dr}, \cite{CP}), we thus obtain a coboundary 
Lie bialgebra structure, given by $\delta(X)=\operatorname{ad}_X(r)$, $X\in\Gh$, where 
$\operatorname{ad}_X(a\otimes b)=[X,a]\otimes b+a\otimes[X,b]$.  We can verify the results 
of the proposition by straightforward computation.
\end{proof}

Corresponding to the cobracket $\delta:\Gh\to\Gh\wedge\Gh$ given above, we can define 
a Lie bracket on the dual space ${\Gh}^*$ of $\Gh$ by $[\ ,\ ]=\delta^*:
{\Gh}^*\wedge{\Gh}^*\to{\Gh}^*$.  That is, $[\mu,\nu]$ is defined by
\begin{equation}\label{(delta^*)}
\bigl\langle[\mu,\nu],X\bigr\rangle=\bigl\langle\delta^*(\mu\otimes\nu),
X\bigr\rangle=\bigl\langle\mu\otimes\nu,\delta(X)\bigr\rangle,
\end{equation}
where $X\in\Gh$, $\mu,\nu\in{\Gh}^*$, and $\langle\ ,\ \rangle$ is the dual pairing between 
${\Gh}^*$ and $\Gh$.  It turns out that the ``dual'' Lie algebra structure on $\Gh^*$ coincides 
with the Lie algebra $\Gg$ described in the previous section.  See the proposition below 
(the proof is straightforward):

\begin{prop}\label{dualLiealgebrag}
Let $\Gg={\Gh}^*$ be spanned by $\mathbf{p}_i,\mathbf{q}_i(i=1,\dots,n),\mathbf{r}$, which form 
the dual basis of $\mathbf{x}_i,\mathbf{y}_i (i=1,\dots,n),\mathbf{z}$.  On $\Gg$, the Lie algebra 
relations can be defined by equation \eqref{(delta^*)}.  Then we have:
$$
[\mathbf{p}_i,\mathbf{p}_j]=0,\quad[\mathbf{q}_i,\mathbf{q}_j]=0,\quad
[\mathbf{p}_i,\mathbf{q}_j]=0,\quad[\mathbf{p}_i,\mathbf{r}]=
\sum_{k=1}^nJ_{ik}\mathbf{q}_k,\quad[\mathbf{q}_i,\mathbf{r}]=0,
$$
for $i,j=1,\dots,n$.  This is the Poisson dual of the Lie bialgebra $(\Gh,\delta)$.
\end{prop}

Comparing the result of Proposition \ref{dualLiealgebrag} with the definition of the Lie algebra 
structure on $\Gg$ given in Section 2.1, we can see clearly that they are indeed the same.  This 
re-interpretation of our Lie algebra $\Gg$ means that $\Gg$ is actually a Lie bialgebra, being 
a dual Lie bialgebra of $(\Gh,\delta)$.  The cobracket on $\Gg$ is the dual map of the Lie bracket 
on $\Gh$.  A short calculation shows that the cobracket $\theta:\Gg\to\Gg\wedge\Gg$ takes its 
values on the basis vectors of $\Gg$ as follows:
$$
\theta(\mathbf{p}_i)=0,\quad\theta(\mathbf{q}_i)=0,
\quad\theta(\mathbf{r})=\sum_{i=1}^n(\mathbf{p}_i\otimes
\mathbf{q}_i-\mathbf{q}_i\otimes\mathbf{p}_i)=\sum_{i=1}^n
(\mathbf{p}_i\wedge\mathbf{q}_i). 
$$

We thus have the (Poisson dual) Lie bialgebra $(\Gg,\theta)$.  Let us now consider the corresponding 
Poisson--Lie group $G$ and its Poisson bracket.  See Proposition~\ref{PoissonLiegroupG} below. 
As before, we are regarding $p=p_1\mathbf{p}_1+p_2\mathbf{p}_2+\cdots+p_n\mathbf{p}_n$, and similarly 
for the other variables.

\begin{prop}\label{PoissonLiegroupG}
Let $G$ be the $(2n+1)$-dimensional Lie group, together with the multiplication law
$$
(p,q,r)(p',q',r')=\left(p+p',q+q'+r'\sum_{i,k=1}^n J_{ik}p_i\mathbf{q}_k,r+r'\right).
$$
This gives us the Lie group corresponding to $\Gg$ from Proposition \ref{dualLiealgebrag}. 
The Poisson bracket on $G$ is given by
$$
\{f,g\}(p,q,r)=r\bigl(\beta(x,y')-\beta(x',y)\bigr)+\frac{r^2}{2}\sum_{i,k=1}^n 
J_{ik}(y_ky'_i-y_iy'_k),
$$
for $f,g\in C^{\infty}(G)$.  Here, $df(p,q,r)=(x,y,z)$ and $dg(p,q,r)=(x',y',z')$, 
which are naturally viewed as elements of $\Gh$.
\end{prop}

\begin{proof}
Construction of $G$ from $\Gg$ was already described in equation \eqref{(groupmult)}.  To find 
the expression for the Poisson bracket, we follow the standard procedure \cite{CP}.  See also 
Proposition 2.3 of \cite{BJKjgp}.  A similar computation (for a different Poisson structure) 
can be found in the proof of Theorem 2.2 of \cite{BJKp2}.

First, consider $\operatorname{Ad}:G\to\operatorname{Aut}(\Gg)$, the adjoint representation 
of $G$ on $\Gg$.  Then we look for a map $F:G\to\Gg\wedge\Gg$, that is a group 1-cocycle on 
$G$ for the $\operatorname{Ad}$-representation and whose derivative at the identity element, 
$dF_e$, coincides with $\theta$ above.  Or, $dF_{(0,0,0)}=\theta$.  In general, integrating 
$\theta$ to $F$ is not always easy.  However, in our case, it is not difficult to check that 
the following map $F$ indeed satisfies the requirements above:
$$
F(p,q,r)=r\sum_{i=1}^n(\mathbf{p}_i\wedge\mathbf{q}_i)-\frac
{r^2}{2}\sum_{i,k=1}^n (J_{ik}\mathbf{q}_k\wedge\mathbf{q}_i).
$$

Since we have the 1-cocycle $F$, the Poisson bivector field is then obtained by the right 
translation of $F$.  To compute, suppose $f,g\in C^{\infty}(G)$.  For $(p,q,r)\in G$, 
since $df=df_e$ is the (linear) differential of the scalar-valued map $f$ on $G$, and 
since $\Gg$ is the tangent space of $G$ at its identity $e=(0,0,0)$, we can naturally 
identify $df(p,q,r)$ as an element in $\Gg^*=\Gh$.  Similarly for $dg(p,q,r)$.  So write 
$df(p,q,r)=(x,y,z)$ and $dg(p,q,r)=(x',y',z')$.  Noting that ${R_{(p,q,r)}}_*(\mathbf{p}_i)
=\mathbf{p}_i+r\sum_{k=1}^n J_{ik}\mathbf{q}_k$ and ${R_{(p,q,r)}}_*(\mathbf{q}_i)
=\mathbf{q}_i$ under the right translation, we have:
\begin{align}
\{f,g\}(p,q,r)&=\bigl\langle {R_{(p,q,r)}}_* F(p,q,r),
df(p,q,r)\wedge dg(p,q,r)\bigr\rangle  \notag \\
&=r\bigl(\beta(x,y')-\beta(x',y)\bigr)+\frac{r^2}{2}
\sum_{i,k=1}^n J_{ik}(y_ky'_i-y_iy'_k).  \notag
\end{align}
\end{proof}

We can see from Proposition \ref{PoissonLiegroupG} that we thus have a non-linear Poisson bracket 
on our group $G$.  When $J=O$ (zero matrix), it becomes linear, carrying only the part that comes 
from the Lie algebra structure on $\Gh=\Gg^*$.  Ours is actually a ``cocycle perturbation'' of 
the linear Poisson bracket, as introduced in \cite{BJKp1}.  See Section~3.1 below for further 
discussion.

\section{Deformation quantization of $G$}

Now that we have described our Poisson--Lie group $G$, we wish to construct its quantum group 
counterpart.  The Poisson data should guide our direction of quantization.

In Section 3.2 of \cite{BJKjgp}, we obtained a $C^*$-bialgebra that can be reasonably considered 
as a quantum semigroup corresponding to $G$.  The method was via a ``cocycle bicrossed product'' 
construction, as in \cite{VV} (see also Section 8 of \cite{BS}).  However, the full construction 
of the quantum group was not carried out, for instance the existence proof of an appropriate 
Haar weight.  In addition, it will be desirable to show a more comprehensive relationship between 
the Poisson--Lie group and the quantum group, including the deformation picture.  We will fill in 
these gaps as we review and improve on our quantum group construction.

\subsection{Poisson bracket of the cocycle perturbation type}

The Poisson bracket on $G$, as obtained in Proposition~\ref{PoissonLiegroupG} above, is of 
the ``cocycle perturbation'' type studied in Theorem 2.2 and Theorem 2.3 of \cite{BJKp1}. 
Let us be more specific.

Since we are identifying $G\cong\Gg$, our Poisson bracket on $G$ may be also regarded as 
a Poisson bracket on $\Gg=\Gh^*$, where $\Gh$ is the Heisenberg Lie algebra noted earlier. 
Let $\Gz$ denote the center of $\Gh$, spanned by the basis element $\mathbf{z}\in\Gh$, and 
let us write $\Gq=\Gz^{\bot}\subseteq\Gg$.  Then we may regard the $x,y,x',y'\in\mathbb{R}^n$ 
as elements of $\Gh/\Gz=\operatorname{span}(\mathbf{x}_i,\mathbf{y}_i:i=1,\dots,n)$ and the 
$r\in\mathbb{R}$ as elements of $\Gg/\Gq$.

Consider the vector space $V=C^{\infty}(\Gg/\Gq)$, and give it the trivial 
$U(\Gh/\Gz)$-module structure.  Suggested by the Poisson bracket expression given in 
Proposition~\ref{PoissonLiegroupG}, let $\omega:\Gh/\Gz\times\Gh/\Gz\to V$ be defined by 
\begin{equation}\label{(Liecocycle)}
\omega\bigl((x,y),(x',y');r\bigr)=r\bigl(\beta(x,y')-\beta(x',y)\bigr)+\frac{r^2}{2}
\sum_{i,k=1}^n J_{ik}(y_ky'_i-y_iy'_k).
\end{equation}
Then $\omega$ is clearly a skew-symmetric bilinear map, and is a Lie algebra cocycle for 
$\Gh/\Gz$, trivially since $\Gh/\Gz$ is an abelian Lie algebra.

Meanwhile, with $\Gh/\Gz$ being abelian, the linear (or ``Lie--Poisson'') Poisson bracket on 
$(\Gh/\Gz)^*$ is the trivial one.  Therefore, our Poisson bracket on $\Gh^*$ is essentially 
the sum of the (trivial) linear Poisson bracket on $(\Gh/\Gz)^*$ and the cocycle $\omega$. 
We thus have the following conclusion:

\begin{prop}
Consider the Poisson bracket on $G$, obtained in Proposition~\ref{PoissonLiegroupG}, which 
is also regarded as defined on $\Gg=\Gh^*$.  It is a ``cocycle perturbation'' of the 
linear Poisson bracket on $\Gh^*$, in the sense of \cite{BJKp1}.
\end{prop}

\begin{proof}
The functions in $V=C^{\infty}(\Gg/\Gq)$ can be canonically realized as functions in 
$C^{\infty}(\Gg)$, by the ``pull-back'' using the natural projection of $\Gg$ onto 
$\Gg/\Gq$.  In addition, the elements in $\Gh$ are linear functions on $\Gg$.  We thus 
have $\Gh+V\subseteq C^{\infty}(\Gg)$, whereas $\Gh\cap V=\Gz$.

Meanwhile, the cocycle $\omega$ on $\Gh/\Gz$ (which takes values in $V$) naturally determines 
a Lie bracket on $\Gh/\Gz\oplus V$, by central extension.  Since $\Gh\cap V=\Gz$, we see that 
$\Gh/\Gz\oplus V\cong\Gh+V$, as vector spaces.  Under this spatial isomorphism, we can thus 
transfer the Lie bracket on $\Gh/\Gz\oplus V$ to a Lie bracket on $\Gh+V$, denoted by 
$[\ ,\ ]_{\Gh+V}$.  This Lie bracket is essentially a ``perturbed Lie bracket'' of the Lie bracket 
on $\Gh$.

With $\Gh+V\subseteq C^{\infty}(\Gg)$, we can give an alternative interpretation of our Poisson 
bracket in Proposition~\ref{PoissonLiegroupG}, as follows:
$$
\{f,g\}(\mu)=\bigl[df(\mu),dg(\mu)\bigr]_{\Gh+V}(\mu),
$$
where $\mu\in\Gg$.  Here, $df(\mu),dg(\mu)\in\Gh(\subseteq\Gh+V)$ as shown in the 
proof of Proposition~\ref{PoissonLiegroupG}; the bracket operation in $\Gh+V$ is as described 
in the previous paragraph; and we are regarding an element in $\Gh+V$ as a function contained 
in $C^{\infty}(\Gg)$.  Having come from the ``perturbed Lie bracket'' of the Lie bracket on 
$\Gh$, our (non-linear) Poisson bracket is a ``cocycle perturbation'' of the linear Poissson 
bracket.

For a more detailed discussion, including some technicalities involving the cocycles, 
refer to Theorem 2.2 and Theorem 2.3 of \cite{BJKp1}, and the paragraphs about the theorems.
\end{proof}

\begin{rem}
When $J=O$ (zero matrix), the cocycle $\omega$ given in equation \eqref{(Liecocycle)} becomes:
$$
\omega_{J=O}\bigl((x,y),(x',y');r\bigr)=r\bigl(\beta(x,y')-\beta(x',y)\bigr).
$$
It is a linear function on $\Gg/\Gq$, so we may write it as:
$$
\omega_{J=O}\bigl((x,y),(x',y')\bigr)=\bigl(\beta(x,y')-\beta(x',y)\bigr)\mathbf{z},
$$
where $\mathbf{z}$ is the basis vector spanning $\Gz$.  In other words, $\omega_{J=O}$ 
is a cocycle for $\Gh/\Gz$ having values in $\Gz$.  It determines the Lie bracket on 
$\Gh$, and therefore, it corresponds to the linear (Lie-Poisson) Poisson bracket 
on $\Gh^*$.  What all this means is that the ``perturbation'' in our case is encoded 
by the matrix $J$ and the associated cocycle $\omega$.
\end{rem}

\subsection{The bicrossed product construction}

Since we realized our Poisson bracket as a cocycle perturbation of the linear Poisson bracket, 
we may follow the steps given in Section 3 of \cite{BJKp1} to construct a deformation 
quantization of $\bigl(C_0(G),\{\ ,\ \}\bigr)$.  The method would use the framework of 
{\em twisted crossed product $C^*$-algebras\/}, in the sense of Packer and Raeburn \cite{PR}.  

However, this method, while valid, gives only the deformation at the $C^*$-algebra level. 
Since we are interested in the construction of a quantum group, let us employ a different 
approach, following instead the one given in \cite{BJKjgp}.  This approach is based on 
the ``bicrossed product construction'' of Vaes and Vainerman \cite{VV}, as well as the 
earlier work by Baaj and Skandalis (Section 8 of \cite{BS}).  Clarification of the 
deformation picture will be postponed to Section 3.3 below.

First, from our Lie algebra cocycle $\omega$ given in equation \eqref{(Liecocycle)}, we obtain 
a continuous family of $\mathbb{T}$-valued group cocycles for the Lie group $H/Z$ of $\Gh/\Gz$. 

\begin{prop}\label{groupcocycle}
Fix an element $r\in\Gg/\Gq$, and define the map $\sigma^r:H/Z\times H/Z\to\mathbb{T}$ by
$$
\sigma^r\bigl((x,y),(x',y')\bigr)=\bar{e}\bigl[r\beta(x,y')\bigr]\bar{e}\left[\frac{r^2}{2}
\sum_{i,k=1}^n J_{ik}y_ky'_i\right],
$$
where $e[t]=e^{2\pi it}$, and so $\bar{e}[t]=e^{-2\pi it}$.  Then each $\sigma^r$ is 
a $\mathbb{T}$-valued, normalized group cocycle for $H/Z$.  In addition, $r\mapsto\sigma^r$ 
forms a continuous field of cocycles.
\end{prop}

\begin{proof}
Let $h=(x,y)$, $h'=(x',y')$, $h''=(x'',y'')$ be elements of $H/Z$, which is just an abelian 
group under addition.  We can easily verify the cocycle identity, as follows:
\begin{align}
&\sigma^r(hh',h'')\sigma^r(h,h')   \notag \\
&=\bar{e}\bigl[r\bigl(\beta(x,y'')+\beta(x',y'')+\beta(x,y')\bigr)\bigr]
\bar{e}\left[\frac{r^2}{2}\sum_{i,k=1}^n J_{ik}(y_ky''_i+y'_ky''_i+y_ky'_i)\right]   \notag \\
&=\sigma^r(h,h'h'')\sigma(h',h'').
\notag
\end{align}
We also have: $\sigma^r(h,0)=1=\sigma^r(0,h)$, where $0=(0,0)$ is the identity element of $H/Z$. 
From the definition, the continuity is quite clear.
\end{proof}

\begin{rem}
In general, constructing the group cocycle by ``integrating'' the Lie algebra cocycle is not 
necessarily easy.  For a little more discussion on this matter, see Section 3 of \cite{BJKp1}.
\end{rem} 

Our Poisson bracket from Proposition~\ref{PoissonLiegroupG} and the group cocycle arising from it, 
as obtained in Proposition~\ref{groupcocycle} above, strongly suggest that it will be most convenient 
for us to work with the $(x,y;r)$ variables, where $(x,y)\in H/Z$ and $r\in\Gg/\Gq=\Gh^*/\Gz^{\bot}$. 
Dual space to $H/Z$ is $(\Gh/\Gz)^*=\Gz^{\bot}$, whose elements are the $(p,q)$.  Following this 
observation, we will break our group $G$ into two parts, obtaining the following matched pair 
of groups.

\begin{defn}\label{matchedpair}
Let $G_1$ and $G_2$ be subgroups of $G$, defined by
$$
G_1=\bigl\{(0,0,r):r\in\mathbb{R}\bigr\},\qquad G_2=\bigl\{(p,q,0):p,q\in\mathbb{R}^n\bigr\}.
$$
Clearly, as a space $G\cong G_2\times G_1$, while $G_1$ and $G_2$ are closed subgroups of $G$, 
such that $G_1\cap G_2=\bigl\{(0,0,0)\bigr\}$.  Moreover, any element $(p,q,r)\in G$ can be 
(uniquely) expressed as a product: $(p,q,r)=(0,0,r)(p,q,0)$, with $(0,0,r)\in G_1$ and 
$(p,q,0)\in G_2$.  In other words, the groups $G_1$ and $G_2$ form a {\em matched pair\/} 
(Or, {\em couple assorti\/} as in Section 8 of \cite{BS}.).

From the matched pair $(G_1,G_2)$, we naturally obtain the group actions $\alpha:G_1\times G_2
\to G_2$ and $\gamma:G_2\times G_1\to G_1$, defined by
$$
\alpha_r(p,q):=\left(p,q-r\sum_{i,k=1}^n J_{ik}p_i\mathbf{q}_k\right),
\qquad\gamma_{(p,q)}(r):=r.
$$
Here we are using the obvious identification of $(p,q)$ with $(p,q,0)$, and similarly for $r$ 
and $(0,0,r)$.  Note that these actions are defined so that we have: $\bigl(\alpha_r(p,q)\bigr)
\bigl(\gamma_{(p,q)}(r)\bigr)=\left(p,q-r\sum_{i,k} J_{ik}p_i\mathbf{q}_k,0\right)(0,0,r)
=(p,q,r)$.
\end{defn}

Let us now convert the information we obtained so far into the language of Hilbert space operators 
and operator algebras.  Recall that we chose a Lebesgue measure on $H(=\Gh)$, which is the 
Haar measure for $H$.  On $G(=\Gg=\Gh^*)$, which is considered as the dual vector space of $H$, 
we give the dual Lebesgue measure.  This will be also the Haar measure for $G$.  These measures 
are chosen so that the Fourier transform becomes the unitary operator (from $L^2(H)$ to $L^2(G)$), 
and the Fourier inversion theorem holds.  Similarly, ``partial'' Fourier transform can be considered, 
for instance, between functions in the $(p,q;r)$ variables and those in the $(x,y;r)$ variables.  See 
Remark 1.7 of \cite{BJKp2}.

First, we define the multiplicative unitary operators $X\in{\mathcal B}\bigl(L^2(G_1\times G_1)\bigr)$ 
and $Y\in{\mathcal B}\bigl(L^2(G_2\times G_2)\bigr)$, associated with the groups $G_1$ and $G_2$. 
See \cite{BS}.  Namely, define:
$$
X\xi(r;r')=\xi(r+r';r'),\qquad Y\zeta(p,q;p',q')
=\zeta(p+p',q+q';p',q'),
$$
for $\xi\in L^2(G_1\times G_1)$ and $\zeta\in L^2(G_2\times G_2)$.  By Fourier transform, 
${\mathcal F}:L^2(G_2)\cong L^2(H/Z)$, the operator $Y$ can be also expressed as an operator 
in ${\mathcal B}\bigl(L^2(H/Z\times H/Z)\bigr)$, in the $(x,y)$ variables.  In other words, 
for convenience, we will regard ${\mathcal F}^{-1}Y{\mathcal F}$ as same as $Y$.  We then have:
$$
Y\zeta(x,y;x',y')=\zeta(x,y;x'-x,y'-y),\quad\zeta\in L^2(H/Z).
$$

\begin{rem}
By the theory of multiplicative unitary operators (see \cite{BS}), the operator $X$ determines 
the (mutually dual) $C^*$-bialgebras $C_0(G_1)$ and $C^*(G_1)$, and similarly, the operator 
$Y$ determines the $C^*$-bialgebras $C_0(G_2)$ and $C^*(G_2)$.  Working with the $(x,y)$ variables, 
by the Fourier transform, we have: $C_0(G_2)\cong C^*(H/Z)$ and $C^*(G)\cong C_0(H/Z)$.  Since 
the groups are abelian, all the computations are quite simple. 

For convenience, a function $f\in C_0(G_1)$ will be considered same as the multiplication 
operator $L_f\in{\mathcal B}\bigl(L^2(G_1)\bigr)$, defined by $L_f\xi(r)=f(r)\xi(r)$.  Similar 
for $g\in C_0(G_2)$, which will be also considered as the multiplication operator $\lambda_g
\in{\mathcal B}\bigl(L^2(G_2)\bigr)$.  In the $(x,y)$ variables, this is equivalent to 
saying that for $g\in C_c(H/Z)\subseteq C^*(H/Z)$, the operator $L_g\in{\mathcal B}
\bigl(L^2(H/Z)\bigr)$ is such that for $\zeta\in L^2(H/Z)$, we have: $L_g\zeta(x,y)
=\int g(\tilde{x},\tilde{y})\zeta(x-\tilde{x},y-\tilde{y})\,d\tilde{x}d\tilde{y}$.
\end{rem}

Next, we try to encode the actions $\alpha$ and $\gamma$ into an operator.  Note that at the level 
of the $C^*$-algebras $C_0(G_1)$ and $C_0(G_2)$, the group actions $\alpha$ and $\gamma$ we defined 
above (though $\gamma$ is trivial) are expressed as coactions $\alpha:C_0(G_2)\to M\bigl(C_0(G_2)
\otimes C_0(G_1)\bigr)$ and $\gamma:C_0(G_1)\to M\bigl(C_0(G_2)\otimes C_0(G_1)\bigr)$, given by
\begin{align}
\alpha(g)(p,q;r)&=g\left(p,q-r\sum_{i,k=1}^n J_{ik}p_i\mathbf{q}_k\right)
=g\bigl(\alpha_r(p,q)\bigr),  \notag \\
\gamma(f)(p,q;r)&=f(r)=f\bigl(\gamma_{(p,q)}(r)\bigr).  \notag
\end{align}
The coactions $\alpha$ and $\gamma$ can be realized using a certain unitary operator $Z$, 
as follows:

\begin{prop}\label{Z}
Let $Z\in{\mathcal B}\bigl(L^2(G)\bigr)={\mathcal B}\bigl(L^2(G_2\times G_1)\bigr)$ be defined by
$$
Z\xi(p,q;r)=\xi\left(p,q-r\sum_{i,k=1}^n J_{ik}p_i\mathbf{q}_k;r\right).
$$
Then we have, for $g\in C_0(G_2)$ and $f\in C_0(G_1)$,
$$
Z(\lambda_g\otimes1)Z^*=(\lambda\otimes L)\bigl(\alpha(g)\bigr),
\qquad Z(1\otimes L_f)Z^*=(\lambda\otimes L)\bigl(\gamma(f)\bigr).
$$
\end{prop}

\begin{proof}
The computations are straightforward.
\end{proof}

\begin{rem}
By using the operator realizations $g=\lambda_g$ and $f=L_f$, as well as 
$\alpha(g)=(\lambda\otimes L)\bigl(\alpha(g)\bigr)$ and $\gamma(f)
=(\lambda\otimes L)\bigl(\gamma(f)\bigr)$, we may simply write the above 
result as: $\alpha(g)=Z(g\otimes1)Z^*$ and $\gamma(f)=Z(1\otimes f)Z^*$. 
\end{rem}

As indicated above, it is more convenient to work with the $(x,y;r)$ variables.  So from now on, 
consider the Hilbert space ${\mathcal H}:=L^2(H/Z\times G_1)$, consisting of the $L^2$-functions 
in the $(x,y;r)$ variables.  Since we know, by the Fourier transform, ${\mathcal F}:L^2(G_2)
\cong L^2(H/Z)$, that $C_0(G_2)\cong C^*(H/Z)$, we may as well regard the coactions $\alpha$ and 
$\gamma$ to be on $C^*(H/Z)$ and $C_0(G_1)$ (In that case, the definitions of $\alpha$ and $\gamma$ 
should be modified accordingly.).  The operator $Z\in{\mathcal B}\bigl(L^2(G_2\times G_1)\bigr)$ 
of Proposition~\ref{Z} then becomes:
\begin{align}
&({\mathcal F}^{-1}\otimes1)Z({\mathcal F}\otimes1)\xi(x,y;r)  \notag \\
&=\int\xi(\tilde{x},\tilde{y};r)\bar{e}\left[p\cdot\tilde{x}
+\left(q-r\sum_{i,k}J_{ik}p_i\mathbf{q}_k\right)\cdot\tilde{y}\right]
e[p\cdot x+q\cdot y]\,d\tilde{x}d\tilde{y}dpdq  \notag \\
&=\int\xi(\tilde{x},\tilde{y};r)\bar{e}\left[p\cdot\left(\tilde{x}-x-r\sum_{i,k}J_{ik}\tilde{y}_k
\mathbf{x}_i\right)\right]\bar{e}\big[q\cdot(\tilde{y}-y)\bigr]\,d\tilde{x}d\tilde{y}dpdq  \notag \\
&=\xi\left(x+r\sum_{i,k=1}^nJ_{ik}y_k\mathbf{x}_i,y;r\right).
\notag 
\end{align}
Here, in the second equality, we used the fact that $\left(r\sum_{i,k}J_{ik}p_i\mathbf{q}_k\right)
\cdot\tilde{y}=r\sum_{i,k}J_{ik}p_i\tilde{y}_k=p\cdot\left(r\sum_{i,k}J_{ik}\mathbf{x}_i\tilde{y}_k\right)$. 
And, in the last equality, the Fourier inversion theorem was used.  From now on, for convenience, 
we will regard the operator $Z\in{\mathcal B}({\mathcal H})$ to mean the operator $({\mathcal F}^{-1}
\otimes1)Z({\mathcal F}\otimes1)$ above.

As indicated in Section 8 of \cite{BS}, the matched pair, $(G_1,G_2)$ together with the actions 
$\alpha$ and $\gamma$, determines a multiplicative unitary operator.  This is shown in part (1) 
of the following proposition.  However, this only comes from the group structure on $G$, and 
not its Poisson structure.  So it will not suffice for our purposes.  In our case, we actually 
need to go a little further, and introduce a certain cocycle term $\Theta$.  The definition of 
$\Theta$ comes directly from the Poisson bracket, given in Proposition~\ref{PoissonLiegroupG} 
(see also Proposition~\ref{groupcocycle}).  Our multiplicative unitary operator, incorporating 
both the matched pair and the cocycle, is obtained in part (2) of the following proposition. 
Proposition~\ref{V_Theta} below is none other than Proposition~3.12 in \cite{BJKjgp}.

\begin{prop}\label{V_Theta}
\begin{enumerate}
  \item Define the unitary operator $V\in{\mathcal B}({\mathcal H}\otimes{\mathcal H})
={\mathcal B}\bigl(L^2(H/Z\times G_1\times H/Z\times G_1)\bigr)$, by 
$V=(Z_{12}X_{24}Z^*_{12})Y_{13}$, using the standard leg notation.  It is multiplicative, 
and it determines the two $C^*$-algebras:
$$
A_V\cong C_0(G_1)\rtimes_{\gamma}(H/Z)\quad{\text {and }}\quad
\hat{A}_V\cong C_0(H/Z)\rtimes_{\alpha}G_1.
$$
They are actually (mutually dual) $C^*$-bialgebras, whose comultiplications are given by 
$\Delta_V(a)=V(a\otimes1)V^*$ for $a\in A_V$, and $\hat{\Delta}_V(b)=V^*(1\otimes b)V$ 
for $b\in\hat{A}_V$.
  \item Let $\Theta(x,y,r;x',y',r'):=\bar{e}\bigl[r'\beta(x,y')\bigr]\bar{e}\left[\frac{{r'}^2}{2}
\sum_{i,k}J_{ik}y_ky'_i\right]$, considered as a unitary operator contained in ${\mathcal B}
({\mathcal H}\otimes{\mathcal H})$.  Then the function $\Theta$ is a cocycle for $V$.  In this 
way, we obtain a multiplicative unitary operator $V_{\Theta}:=V\Theta\in{\mathcal B}({\mathcal H}
\otimes{\mathcal H})$.
Specifically,
\begin{align}
&V_{\Theta}\xi(x,y,r;x',y',r')  \notag \\
&=e\left[\frac{{r'}^2}{2}\sum_{i,k}J_{ik}y_k(y'_i-y_i)\right]\bar{e}\bigl[r'\beta(x,y'-y)\bigr] 
\notag \\
&\qquad\xi\left(x-r'\sum_{i,k}J_{ik}y_k\mathbf{x}_i,y,r+r';
x'-x+r'\sum_{i,k}J_{ik}y_k\mathbf{x}_i,y'-y,r'\right).
\notag
\end{align}
\end{enumerate}
The $C^*$-bialgebras associated with $V_{\Theta}$ are:
$$
S\cong C_0(G_1)\rtimes_{\gamma}^{\sigma}(H/Z),\quad{\text {and }}\quad
\hat{S}\cong C_0(H/Z)\rtimes_{\alpha} G_1,
$$
together with the comultiplications $\Delta(a):=V_{\Theta}(a\otimes1){V_{\Theta}}^*$ for 
$a\in S$, and $\hat{\Delta}(b):={V_{\Theta}}^*(1\otimes b)V_{\Theta}$ for $b\in\hat{S}$. 
Here, $\sigma:r\mapsto\sigma^r$ is a continuous field of cocycles such that 
$\sigma^r\bigl((x,y),(x',y')\bigr)=\bar{e}\left[\frac{{r}^2}{2}\sum_{i,k}J_{ik}y_ky'_i\right]
\bar{e}\bigl[r\beta(x,y')\bigr]$.
\end{prop}

\begin{proof}
(1). The choice of the operator $V$, arising from the matched pair $(G_1,G_2)$, is suggested by 
Section 8 of \cite{BS}.  As noted above, the operators $X$ and $Y$ encode the groups $G_1$ and $G_2$, 
while the actions $\alpha$ and $\gamma$ are encoded by the operator $Z$.  The multiplicativity 
of $V$ is just a simple consequence of the fact that $G$ is a group.  From the general theory 
of multiplicative unitary operators \cite{BS}, we thus obtain the (mutually dual) $C^*$-bialgebras 
$A_V$ and $\hat{A}_V$ by considering the ``left [and right] slices'' of $V$.  The proof for the 
characterizations of the two $C^*$-algebras is also straightforward, and will be skipped.

(2). The function $\Theta$ is a cocycle for $V$, since $V_{\Theta}$ is also multiplicative.  The 
verification of the pentagon equation, $W_{12}W_{13}W_{23}=W_{23}W_{12}$ for $W=V_{\Theta}$, is 
straightforward.

As usual, the $C^*$-bialgebras associated with $V_{\Theta}$ are obtained by
\begin{align}
S&=\bigl\{(\omega\otimes\operatorname{id}_{\mathcal H})
(V_{\Theta}):\omega\in{\mathcal B}({\mathcal H})_*\bigr\}
\bigl(\subseteq{\mathcal B}({\mathcal H})\bigr),  \notag \\
\hat{S}&=\bigl\{(\operatorname{id}_{\mathcal H}\otimes\omega)
(V_{\Theta}):\omega\in{\mathcal B}({\mathcal H})_*\bigr\}
\bigl(\subseteq{\mathcal B}({\mathcal H})\bigr).  \notag
\end{align}
Their comultiplications are defined in the standard way, via the multiplicative unitary operator.
For the verification of the $C^*$-algebra realizations of $S$ and $\hat{S}$ as twisted crossed 
product $C^*$-algebras above, refer to the proof of Proposition~3.12 of \cite{BJKjgp}.  Since the 
groups $G_1$ and $H/Z$ are amenable (being abelian), the notions of the reduced and full (twisted) 
crossed products coincide.
\end{proof}

Observe that the cocycle term for the twisted crossed product $C^*$-algebra follows directly 
from the underlying Poisson structure.  In fact, the $C^*$-bialgebra $(S,\Delta)$ is essentially 
a ``quantized $C^*(H)$'' or a ``quantized $C_0(G)$''.  For instance, if $J\equiv0$, then it is 
not difficult to show that $S\cong C^*(H)$ as an algebra.  In addition, see Section 3.3 below 
for the clarification that $(S,\Delta)$ is indeed a deformation quantization of $C_0(G)$, in 
the direction of its Poisson bracket.

\subsection{$(S,\Delta)$ as a deformation quantization of $\bigl(G,\{\ ,\ \}\bigr)$}

We constructed above a $C^*$-bialgebra $(S,\Delta)$, by means of the multiplicative unitary 
operator $V_{\Theta}$.  There are strong indications that $(S,\Delta)$ should be an appropriate 
quantum counterpart to the Poisson--Lie group $\bigl(G,\{\ ,\ \}\bigr)$.  In this subsection, 
we make this picture clearer, by showing that the $C^*$-algebra $S$ is a (strict) deformation 
quantization of $C_0(G)$, in the sense of Rieffel \cite{Rf1}, \cite{Rf4}.

Let us analyze the $C^*$-algebra $S$ a bit.  For $f\in C_c(G)$, we can carry it into 
a function of the $(x,y,r)$ variables by the (partial) Fourier transform: $f\mapsto 
f^{\vee}\in C_0(H/Z\times G_1)$, where 
$f^{\vee}(x,y,r)=\int f(p,q,r)e[p\cdot x+q\cdot y]\,dpdq$.
Considering this, let us define the operator $L_f\in{\mathcal B}({\mathcal H})$ by 
\begin{equation}\label{(Lrep)}
L_f\xi(x,y,r):=\int f^{\vee}(\tilde{x},\tilde{y},r)
\sigma^r\bigl((\tilde{x},\tilde{y}),(x-\tilde{x},y-\tilde{y})\bigr)
\xi(x-\tilde{x},y-\tilde{y},r)\,d\tilde{x}d\tilde{y},
\end{equation}
where $\sigma$ is the cocycle given in Proposition~\ref{V_Theta}\,(2). 

\begin{rem}
If $\sigma\equiv1$, the above representation $L:C_c(G)\ni f\mapsto L_f
\in{\mathcal B}({\mathcal H})$ is equivalent (by the partial Fourier transform) to 
$\lambda\otimes L:C_c(G_2\times G_1)\mapsto{\mathcal B}\bigl(L^2(G_2\times G_1)\bigr)
={\mathcal B}\bigl(L^2(G)\bigr)$, with the representations $L$ and $\lambda$ on 
$C_0(G_1)$ and $C_0(G_2)$ defined earlier.  See also Theorem~\ref{DQtheorem}\,(1). 
Since there is no worry about confusion, we chose to use the same name $L$ for our 
(extended) representation.
\end{rem}

By the result of Proposition~\ref{V_Theta}\,(2), it is clear that $S\cong\overline
{L\bigl(C_c(G)\bigr)}^{\|\ \|}$, as a $C^*$-algebra.
What all this means is that we do have a (deformed) ${}^*$-algebra structure at the level of 
the functions on $G$, inherited from the ${}^*$-algebra structure on $S$.  To be more precise, 
let ${\mathcal A}={\mathcal S}_{3c}(G)\,\bigl(\subseteq C_0(G)\bigr)$, the space of Schwartz 
functions having compact support in the $r$-variable.  It is slightly larger than $C_c^{\infty}(G)$, 
and is the image under the partial Fourier transform, ${}^\wedge$, of the space ${\mathcal S}_{3c}
(H/Z\times G_1)\,\bigl(\subseteq C_0(H/Z\times G_1)\bigr)$.  On ${\mathcal A}$, we can define 
the deformed product, $\times$, by 
\begin{align}\label{(deformedproduct)}
(f\times g)(p,q,r)&=(f^{\vee}\ast_{\sigma}g^{\vee})^{\wedge}(p,q,r)  \notag \\
&=\int\bar{e}[p\cdot x+q\cdot y]f^{\vee}(\tilde{x},\tilde{y},r)g^{\vee}(x-\tilde{x},y-\tilde{y},r) \\
&\qquad\bar{e}\left[\frac{{r}^2}{2}\sum_{i,k}J_{ik}\tilde{y}_k(y_i-\tilde{y}_i)\right]
\bar{e}\bigl[r\beta(\tilde{x},y-\tilde{y})\bigr]\,d\tilde{x}d\tilde{y}dxdy.
\notag
\end{align}
Using the definitions of $f^{\vee}$ and $g^{\vee}$, together with the Fourier inversion theorem, 
this expression becomes:
\begin{equation}\label{(deformedproductpqr)}
(f\times g)(p,q,r)=\int\bar{e}\bigl[(q-\tilde{q})\cdot y\bigr]
f\left(p+ry,q+\frac{r^2}{2}\sum_{i,k}J_{ik}y_i\mathbf{q}_k,r\right)g(p,\tilde{q},r)\,d\tilde{q}dy.
\end{equation}
Similarly, the involution on ${\mathcal A}$ is given by
\begin{align}\label{(involution)}
&f^*(p,q,r)=\bigl((f^{\vee})^*\bigr)^{\wedge}(p,q,r)   \notag \\
&=\int\overline{f(\tilde{p},\tilde{q},r)}\bar{e}\bigl[(p-\tilde{p})\cdot x+(q-\tilde{q})\cdot y\bigr]
\bar{e}\left[\frac{{r}^2}{2}\sum_{i,k}J_{ik}y_iy_k\right]
\bar{e}\bigl[r\beta(x,y)\bigr]\,d\tilde{p}d\tilde{q}dxdy.
\end{align}
Clearly, the ${}^*$-algebra $({\mathcal A},\times,{}^*)$ is a pre-$C^*$-algebra, together with 
the $C^*$-norm $f\mapsto\|L_f\|$.  Here, the representation $L$ is just as in 
equation~\eqref{(Lrep)}, having been extended to ${\mathcal A}$.  By construction, 
we know that $S\cong\overline{L({\mathcal A})}^{\|\ \|}$.

To show that the $C^*$-algebra $S$ is a deformation of $\bigl(C_0(G),\{\ ,\ \}\bigr)$, let us now 
introduce the deformation parameter $\hbar$.  We will follow the general procedure given in 
Theorem~3.4 of \cite{BJKp1}.  In our case, with the group $H/Z$ being abelian, it does not 
need to vary and we only need to incorporate the parameter $\hbar$ to the cocycle $\sigma$. 
Namely, consider the cocycle $\sigma_{\hbar}:r\mapsto\sigma_{\hbar}^r$, given by
$$
\sigma_{\hbar}^r\bigl((x,y),(x',y')\bigr)=\bar{e}\left[\frac{\hbar{r}^2}{2}\sum_{i,k}J_{ik}y_ky'_i\right]
\bar{e}\bigl[\hbar r\beta(x,y')\bigr].
$$
Then in exactly the same way as in equations~\eqref{(deformedproductpqr)} and \eqref{(involution)}, 
but by using the cocycle $\sigma_{\hbar}$ instead, we can construct on the function space ${\mathcal A}$ 
the deformed multiplication $\times_{\hbar}$ and the involution ${}^{*_{\hbar}}$.  As before, 
each $\bigl({\mathcal A},\times_{\hbar},{}^{*_{\hbar}}\bigr)$ is a pre-$C^*$-algebra: Similarly 
to equation~\eqref{(Lrep)}, the functions $f\in{\mathcal A}$ can be regarded as operators, with 
the operator norm now denoted by $\|\ \|_{\hbar}$.  Let us define $S_{\hbar}$ as the $C^*$-completion 
of $\bigl({\mathcal A},\times_{\hbar},{}^{*_{\hbar}}\bigr)$, under $\|\ \|_{\hbar}$.  Using these 
ingredients, we can now describe the deformation quantization picture.

\begin{theorem}\label{DQtheorem}
Recall the Poisson bracket $\{\ ,\ \}$ on $G$, from Proposition~\ref{PoissonLiegroupG}. Let ${\mathcal A}
={\mathcal S}_{3c}(G)$ be the (dense) subspace of $C_0(G)$ as defined above.  For each $\hbar\in\mathbb{R}$, 
define on ${\mathcal A}$ the deformed multiplication, $\times_{\hbar}$, and the involution, ${}^{*_{\hbar}}$, 
as in the previous paragraph, together with the corresponding $C^*$-norm $\|\ \|_{\hbar}$.  Then we have:
\begin{enumerate}
  \item For $\hbar=0$, the operations $\times_{\hbar}$, ${}^{*_{\hbar}}$ are exactly the pointwise product 
and the complex conjugation on ${\mathcal A}\,\bigl(\subseteq C_0(G)\bigr)$.  Also $S_{\hbar=0}\cong C_0(G)$, 
as a $C^*$-algebra.
  \item The $C^*$-algebras $\{S_{\hbar}\}_{\hbar\in\mathbb{R}}$ form a continuous field of $C^*$-algebras. 
In particular, the map $\hbar\mapsto\|f\|_{\hbar}$ is continuous for any $f\in{\mathcal A}$.
  \item For any $f,g\in{\mathcal A}$ and $(p,q,r)\in G$, we have the following pointwise convergence:
$$
\frac{1}{\hbar}(f\times_{\hbar}g-g\times_{\hbar}f)(p,q,r)\,\longrightarrow\,\frac{i}{2\pi}\{f,g\}(p,q,r),
$$
as $\hbar\to0$.
  \item The convergence in (3) is actually stronger.  In fact, for $f,g\in{\mathcal A}$, we have:
\begin{equation}\label{(correspondence)}
\lim_{\hbar\to0}\left\|\frac{f\times_{\hbar}g-g\times_{\hbar}f}{i\hbar}-\frac{1}{2\pi}\{f,g\}\right\|_{\hbar}=0.
\end{equation}
\end{enumerate}
All this means that the ${}^*$-algebras $\bigl({\mathcal A},\times_{\hbar},{}^{*_{\hbar}}\bigr)_{\hbar\in\mathbb{R}}$ 
provide a ``strict deformation quantization'' (in the sense of Rieffel \cite{Rf1}, \cite{Rf4}) of 
${\mathcal A}\,\bigl(\subseteq C_0(G)\bigr)$, in the direction of the Poisson bracket $(1/2\pi)\{\ ,\ \}$.
\end{theorem}

\begin{proof}
This is a special case of Theorem~3.4 of \cite{BJKp1}.  But, we will still carry out some main aspects 
of the proof.

(1). If $\hbar=0$, the cocycle term becomes $\sigma\equiv1$, and can be ignored.  By the Fourier inversion 
theorem, equations~\eqref{(deformedproductpqr)} and \eqref{(involution)} thus become:
\begin{align}
(f\times g)(p,q,r)&=\int\bar{e}\bigl[(q-q')\cdot\tilde{y}\bigr]f(p,q',r)g(p,q,r)\,dq'd\tilde{y}
=f(p,q,r)g(p,q,r), \notag \\
f^*(p,q,r)&=\int\overline{f(\tilde{p},\tilde{q},r)}
\bar{e}\bigl[(p-\tilde{p})\cdot x+(q-\tilde{q})\cdot y\bigr]\,d\tilde{p}d\tilde{q}dxdy
=\overline{f(p,q,r)}.  \notag
\end{align}
It is also easy to see that $S_{\hbar=0}\cong C_0(G)$, with its sup-norm as the $C^*$-norm.  When $\hbar=1$, 
we would recover the $C^*$-algebra $S$ of Proposition~\ref{V_Theta}\,(2).

(2). As for the $C^*$-algebras $\{S_{\hbar}\}_{\hbar\in\mathbb{R}}$ forming a continuous field of 
$C^*$-algebras, note that each $S_{\hbar}$ is really a twisted crossed product $C^*$-algebra of 
an abelian group $H/Z$, namely $S_{\hbar}\cong C_0(G_1)\rtimes^{\sigma_{\hbar}}(H/Z)$, and only the 
cocycle $\sigma_{\hbar}$ is being changed as the parameter $\hbar$ varies.  Therefore, for each $\hbar$, 
the ``amenability condition'' holds, meaning that the notions of the ``full'' and the ``reduced'' crossed 
product $C^*$-algebras coincide.  In \cite{RfC*field}, using the universal property of the full 
$C^*$-algebras and also taking advantage of the property of the reduced $C^*$-algebras that one is 
able to work with their specific representations, Rieffel gave an answer to the problem of the continuity 
of certain field of crossed product $C^*$-algebras: In short, under suitable conditions, Rieffel has 
shown that the field of ``full'' crossed product $C^*$-algebras is upper semi-continuous, while the 
field of ``reduced'' crossed product $C^*$-algebras is lower semi-continuous.  Our case is simpler 
than the general case, and with the amenability at hand, it follows that our field of $C^*$-algebras 
$\{S_{\hbar}\}_{\hbar\in\mathbb{R}}$ is in fact continuous.

(3). For $f\in{\mathcal A}$, by Fourier inversion theorem, we can write it as
$$
f(p,q,r)=\int({\mathcal F}^{-1}f)(\tilde{x},\tilde{y},\tilde{z})\bar{e}[p\cdot\tilde{x}
+q\cdot\tilde{y}+r\tilde{z}]\,d\tilde{x}d\tilde{y}d\tilde{z}.
$$
Since $\bar{e}[t]=e^{-2\pi it}$, we thus have:
$$
df(p,q,r)=(-2\pi i)\int({\mathcal F}^{-1}f)(\tilde{x},\tilde{y},\tilde{z})
\bar{e}[p\cdot\tilde{x}+q\cdot\tilde{y}+r\tilde{z}]\mathbf{X}\,d\tilde{x}d\tilde{y}d\tilde{z},
$$
where $\mathbf{X}=(\tilde{x},\tilde{y},\tilde{z})$.  Therefore, for $f,g\in{\mathcal A}$, the Poisson 
bracket from Proposition~\ref{PoissonLiegroupG} becomes:
\begin{align}
&\{f,g\}(p,q,r)=(-4\pi^2)\int({\mathcal F}^{-1}f)(\tilde{x},\tilde{y},\tilde{z})
({\mathcal F}^{-1}g)(\tilde{\tilde{x}},\tilde{\tilde{y}},\tilde{\tilde{z}})
\notag \\
&\qquad\qquad\qquad\qquad\qquad\left[r\bigl(\beta(\tilde{x},\tilde{\tilde{y}})
-\beta(\tilde{\tilde{x}},\tilde{y})\bigr)+\frac{r^2}{2}\sum_{i,k=1}^nJ_{ik}(\tilde{y}_k\tilde{\tilde{y}}_i
-\tilde{y}_i\tilde{\tilde{y}}_k)\right]  \notag \\
&\qquad\qquad\qquad\qquad\qquad\bar{e}\bigl[p\cdot(\tilde{x}+\tilde{\tilde{x}})
+q\cdot(\tilde{y}+\tilde{\tilde{y}})+r(\tilde{z}+\tilde{\tilde{z}})]\,d\tilde{x}d\tilde{y}d\tilde{z}
d\tilde{\tilde{x}}d\tilde{\tilde{y}}d\tilde{\tilde{z}}.
\notag
\end{align}
In the $(x,y,r)$ variables, by using the partial Fourier transform, this can be re-written as
\begin{align}\label{(pb)}
&\{f,g\}(p,q,r)  \notag \\
&=(-4\pi^2)\int f^{\vee}(\tilde{x},\tilde{y},r)g^{\vee}(\tilde{\tilde{x}},\tilde{\tilde{y}},r)
\bar{e}\bigl[p\cdot(\tilde{x}+\tilde{\tilde{x}})+q\cdot(\tilde{y}+\tilde{\tilde{y}})\bigr]  \\
&\qquad\qquad\ \ \left[r\bigl(\beta(\tilde{x},\tilde{\tilde{y}})-\beta(\tilde{\tilde{x}},\tilde{y})\bigr)
+\frac{r^2}{2}\sum_{i,k=1}^nJ_{ik}(\tilde{y}_k\tilde{\tilde{y}}_i-\tilde{y}_i\tilde{\tilde{y}}_k)
\right]\,d\tilde{x}d\tilde{y}d\tilde{\tilde{x}}d\tilde{\tilde{y}}.
\notag
\end{align}

Meanwhile, let us re-write the deformed product, $f\times_{\hbar}g$ for $f,g\in{\mathcal A}$, 
in a more symmetric form.  Basically, we start from the definition given in equation 
\eqref{(deformedproduct)}, together with the adjustment in the cocycle term incorporating 
the parameter $\hbar$.  Perform the change-of-variables: $x-\tilde{x}\mapsto\tilde{\tilde{x}}$ 
and $y-\tilde{y}\mapsto\tilde{\tilde{y}}$.  Then we would have:
\begin{align}
(f\times_{\hbar}g)(p,q,r)&=\int\bar{e}\bigl[p\cdot(\tilde{x}+\tilde{\tilde{x}})
+q\cdot(\tilde{y}+\tilde{\tilde{y}})\bigr]f^{\vee}(\tilde{x},\tilde{y},r)
g^{\vee}(\tilde{\tilde{x}},\tilde{\tilde{y}},r)  \notag \\
&\qquad\bar{e}\bigl[\hbar r\beta(\tilde{x},\tilde{\tilde{y}})\bigr]\bar{e}\left[\frac{\hbar r^2}{2}
\sum_{i,k}J_{ik}\tilde{y}_k\tilde{\tilde{y}}_i\right]\,d\tilde{x}d\tilde{y}d\tilde{\tilde{x}}d\tilde{\tilde{y}}.
\notag
\end{align}
It follows that we have:
\begin{align}\label{(fg-gf)}
&\frac{1}{\hbar}(f\times_{\hbar}g-g\times_{\hbar}f)(p,q,r)   \notag \\
&=\frac{1}{\hbar}\int\bar{e}\bigl[p\cdot(\tilde{x}+\tilde{\tilde{x}})+q\cdot(\tilde{y}+\tilde{\tilde{y}})\bigr]
f^{\vee}(\tilde{x},\tilde{y},r)g^{\vee}(\tilde{\tilde{x}},\tilde{\tilde{y}},r)   \notag \\
&\qquad\left(\bar{e}\bigl[\hbar r\beta(\tilde{x},\tilde{\tilde{y}})\bigr]\bar{e}\left[\frac{\hbar r^2}{2}
\sum_{i,k}J_{ik}\tilde{y}_k\tilde{\tilde{y}}_i\right]-\bar{e}\bigl[\hbar r\beta(\tilde{\tilde{x}},\tilde{y})\bigr]
\bar{e}\left[\frac{\hbar r^2}{2}\sum_{i,k}J_{ik}\tilde{\tilde{y}}_k\tilde{y}_i\right]\right)
\notag \\
&\qquad\quad d\tilde{x}d\tilde{y}d\tilde{\tilde{x}}d\tilde{\tilde{y}}.
\end{align}
In the above, since $\bar{e}[t]=e^{-2\pi i t}=1-2\pi it+\cdots$, we have:
\begin{align}
&\frac{1}{\hbar}\left(\bar{e}\bigl[\hbar r\beta(\tilde{x},\tilde{\tilde{y}})\bigr]\bar{e}\left[\frac{\hbar r^2}{2}
\sum_{i,k}J_{ik}\tilde{y}_k\tilde{\tilde{y}}_i\right]-\bar{e}\bigl[\hbar r\beta(\tilde{\tilde{x}},\tilde{y})\bigr]
\bar{e}\left[\frac{\hbar r^2}{2}\sum_{i,k}J_{ik}\tilde{\tilde{y}}_k\tilde{y}_i\right]\right)
\notag \\
&=(-2\pi i)\left(r\beta(\tilde{x},\tilde{\tilde{y}})+\frac{r^2}{2}\sum_{i,k}J_{ik}\tilde{y}_k\tilde{\tilde{y}}_i
-r\beta(\tilde{\tilde{x}},\tilde{y})-\frac{r^2}{2}\sum_{i,k}J_{ik}\tilde{\tilde{y}}_k\tilde{y}_i\right)
+{\mathcal O}(\hbar).
\notag
\end{align}
Therefore, comparing with equation \eqref{(pb)}, we can readily observe the pointwise convergence:
$$
\frac{1}{\hbar}(f\times_{\hbar}g-g\times_{\hbar}f)(p,q,r)\,\longrightarrow\,\frac{i}{2\pi}\{f,g\}(p,q,r),
$$
as $\hbar\to0$.

(4). In our case, each $S_{\hbar}$ (for $\hbar\ne0$) is isomorphic to the (reduced) twisted crossed 
product $C^*$-algebra $C_0(G_1)\rtimes^{\sigma_{\hbar}}(H/Z)$, and therefore, the $C^*$-norm $\|\ \|_{\hbar}$ 
is dominated by the $L^1$-norm on $L^1\bigl(H/Z,C_0(G_1)\bigr)$.  By the partial Fourier transform in 
the $r(\in G_1)$ variable, this $L^1$-norm is equivalent to the $L^1$-norm on $L^1(H/Z\times Z)=L^1(H)$. 
Even when $\hbar=0$, for which we know $S_{\hbar=0}\cong C_0(G)\cong C^*(H)$ by the Fourier transform, 
it holds that the $C^*$-norm $\|\ \|_{\hbar=0}$ is also dominated by the $L^1$-norm on $L^1(H)$.  All this 
means that to show the norm convergence in equation~\eqref{(correspondence)}, we just need to show the 
convergence with respect to the $L^1$-norm on $L^1(H)$, transferred to ${\mathcal A}\subseteq L^1(G)$ 
by the Fourier transform. 

This can be achieved by Lebesgue's dominated convergence theorem: We already know the pointwise convergence 
in ${\mathcal A}$; While in ${\mathcal A}={\mathcal S}_{3c}(G)$, we are able to find an $L^1$-bound for 
the expressions, $(f\times_{\hbar}g-g\times_{\hbar}f)/{\hbar}-(i/2\pi)\{f,g\}$, since the convergence 
involving the cocycle terms can be controlled in a compact set on which the convergence is uniform.
\end{proof}

\begin{rem}
In the proof of item (2) above, we were aided by the fact that $H/Z$ is abelian.  In general, however, 
the group may not be abelian and may also need to vary (as the parameter value changes) in the definition 
of the $C^*$-algebras $S_{\hbar}$.  This would make the proof of the continuity of $\{S_{\hbar}\}_{\hbar
\in\mathbb{R}}$ more difficult.  Our current example does not have this problem, but refer to the proof 
of Theorem~3.4 in \cite{BJKp1} for a more general situation. 

Meanwhile, as for the proof of the correspondence relation in equation~\eqref{(correspondence)}, 
note that a pointwise convergence result like (3) would be usually sufficient for most of the 
formal power series frameworks, like in the case of a QUE algebra.  But, for our ``strict deformation 
quantization'' framework, we further needed to show the norm convergence, as in (4) above.  See 
\cite{Rf1}, \cite{Rf4} for more general discussions.  The idea for proof of (4) was obtained from 
the one given in \cite{Rf3}, with a small adjustment of restricting things to ${\mathcal S}_{3c}(G)$, 
instead of the space ${\mathcal S}(G)$ of all Schwartz functions on $G$.
\end{rem}

\subsection{The comultiplication on $(S,\Delta)$}

To further strengthen our case that $(S,\Delta)$ is a ``quantized $C_0(G)$'', let us look at the 
comultiplication $\Delta$, which will show that it reflects the group multiplication law on $G$.

\begin{prop}
With the representation $f\mapsto L_f\in{\mathcal B}({\mathcal H})$, $f\in{\mathcal A}$, defined in 
equation~\eqref{(Lrep)}, the comultiplication $\Delta$ from Proposition~\ref{V_Theta} (2) becomes:
$$
\Delta(L_f)=(L\otimes L)_{\Delta f}, 
$$
where $\Delta f\in C_b(G\times G)$ is the function defined by
$$
\bigl(\Delta f\bigr)(p,q,r;p',q',r')=f\left(p+p',q+q'
+r'\sum_{i,k}J_{ik}p_i\mathbf{q}_k,r+r'\right).
$$
\end{prop}

\begin{proof}
Write $L_f=\int({\mathcal F}^{-1}f)(\tilde{x},\tilde{y},\tilde{z})
L_{\tilde{x},\tilde{y},\tilde{z}}\,d\tilde{x}d\tilde{y}d\tilde{z}$, where ${\mathcal F}^{-1}f\in C_c(H)$ 
is the (inverse) Fourier transform of $f$.  Then $L_{\tilde{x},\tilde{y},\tilde{z}}\in{\mathcal B}({\mathcal H})$ 
is such that
$$
L_{\tilde{x},\tilde{y},\tilde{z}}\xi(x,y,r)=\bar{e}[r\tilde{z}]
\sigma^r\bigl((\tilde{x},\tilde{y}),(x-\tilde{x},y-\tilde{y})\bigr)\xi(x-\tilde{x},y-\tilde{y},r).
$$
Comparing with the definition of $L_f$ in equation~\eqref{(Lrep)}, we may regard 
$L_{\tilde{x},\tilde{y},\tilde{z}}=L_F$, where the function $F\in C_b(G)$ is such that: 
$F(p,q,r)=\bar{e}[p\cdot\tilde{x}+q\cdot\tilde{y}+r\tilde{z}]$.  Indeed, $L_{\tilde{x},\tilde{y},\tilde{z}}$ 
is contained in the multiplier algebra $M(S)$.  In a sense, the operators $L_{\tilde{x},\tilde{y},\tilde{z}}$ 
for $(\tilde{x},\tilde{y},\tilde{z})\in H$, form the building blocks for the ``regular representation'' $L$ 
(or equivalently, for $C^*$-algebra $S$).

For $\zeta\in{\mathcal H}$, we have:
\begin{align}
&\bigl(\Delta(L_{\tilde{x},\tilde{y},\tilde{z}})\bigr)
\zeta(x,y,r;x',y',r')=V_{\Theta}(L_{\tilde{x},\tilde{y},
\tilde{z}}\otimes1){V_{\Theta}}^*\zeta(x,y,r;x',y',r')
\notag \\
&=\bar{e}\bigl[(r+r')\tilde{z}\bigr]\bar{e}\left[\frac
{r^2}{2}\sum_{i,k}J_{ik}\tilde{y}_k(y_i-\tilde{y}_i)\right]
\bar{e}\bigl[r\beta(\tilde{x},y-\tilde{y})\bigr]  \notag \\
&\quad\bar{e}\left[\frac{{r'}^2}{2}\sum_{i,k}J_{ik}\tilde{y}_k
(y'_i-\tilde{y}_i)\right]\bar{e}\bigl[r'\beta(\tilde{x},
y'-\tilde{y})\bigr]\bar{e}\left[rr'\sum_{i,k}J_{ik}\tilde{y}_k
(y_i-\tilde{y}_i)\right]  \notag \\
&\quad\zeta\left(x-\tilde{x}-r'\sum_{i,k}J_{ik}\tilde{y}_k
\mathbf{x}_i,y-\tilde{y},r;x'-\tilde{x},y'-\tilde{y},r'
\right).
\notag
\end{align}

Meanwhile, consider $\Delta F\in C_b(G\times G)$, given by
\begin{align}
&(\Delta F)(p,q,r;p',q',r')  \notag \\
&=\bar{e}\left[(p+p')\cdot\tilde{x}+(q+q')\cdot\tilde{y}
+r'\sum_{i,k}J_{ik}p_i\tilde{y}_k+(r+r')\tilde{z}\right].
\notag
\end{align}
Then by a straightforward computation using Fourier inversion theorem, we can see that 
for $\zeta\in{\mathcal H}$:
$$
(L\otimes L)_{\Delta F}\zeta(x,y,r;x',y',r')
=\bigl(\Delta(L_{\tilde{x},\tilde{y},\tilde{z}})\bigr)\zeta(x,y,r;x',y',r').
$$
In other words, $(L\otimes L)_{\Delta F}=\Delta(L_F)$.  Remembering the definitions, 
it follows easily that $\Delta(L_f)=(L\otimes L)_{\Delta f}$ for any $f\in C_c(G)$, 
where $\Delta f$ is as defined above.
\end{proof}

\begin{rem}
This proposition shows that for $f\in C_c(G)$, the comultiplication sends it to 
$\Delta f\in C_b(G\times G)$, such that
$$
(\Delta f)(p,q,r;p',q',r')=f\bigl((p,q,r)(p',q',r')\bigr),
$$
preserving the group multiplication law on $G$ as given in equation~\eqref{(groupmult)}. 
Together with the result of Theorem~\ref{DQtheorem}, this result supports our assertion 
made earlier that $(S,\Delta)$ is a ``quantized $C_0(G)$''.
\end{rem}

At this moment, the $C^*$-bialgebra $(S,\Delta)$ is just a quantum semi-group.  For it 
to be properly considered as a locally compact quantum group, we need further discussions 
on maps like antipode and Haar weight.  This is given in the following section.

Meanwhile, notice the similarity between our example $(S,\Delta)$ above and the one constructed 
by Enock and Vainerman in Section 6 of \cite{EV}.  Looking at the comultiplications and the 
cocycles involved, we see some resemblance.  However, the methods of construction are rather 
different between the two.  In addition, there is another crucial difference.  Namely, 
in the example of \cite{EV}, the underlying von  Neumann algebra is isomorphic to the group 
von Neumann algebra ${\mathcal L}(H)=C^*(H)''$ of $H$.  While in our case, $S$ is isomorphic 
to a ``twisted'' crossed product algebra: Unless $J\equiv0$, the $C^*$-algebra $S$ is not 
isomorphic to $C^*(H)$.

In the author's opinion, the example $(S,\Delta)$ given here has more merit, considering that: 
(1) its Poisson--Lie group counterpart and its multiplicative unitary operator have all been 
obtained; (2) the relationship between the Poisson bracket and the cocycle bicrossed product 
construction of the multiplicative unitary operator have been manifested; (3) as well as that 
the underlying $C^*$-algebra is built on the framework of twisted crossed product algebras 
(more general than ordinary group $C^*$-algebras or group von Neumann algebras).

\section{The quantum group structure}

\subsection{$(S,\Delta)$ is a locally compact quantum group}

We now turn our attention to showing that the $C^*$-bialgebra $(S,\Delta)$ we constructed above 
is indeed a {\em locally compact quantum group\/}, in the precise sense of Kustermans and 
Vaes \cite{KuVa}, \cite{KuVavN}, or that of Masuda, Nakagami, and Woronowicz \cite{MNW}.  We 
could construct the Haar weight and other maps, along the lines of the general results by Van Daele 
\cite{VDHaar}, \cite{VDoamp}.  However, since it can be shown that our example is a case of 
a ``cocycle bicrossed product'' (in the sense of \cite{VV}), it is not really necessary to be overly 
technical.  See Lemma~\ref{matchinglemma} and Theorem~\ref{lcqg} below.

First, recall the matched pair $(G_1,G_2)$ we considered in Definition~\ref{matchedpair}. 
Our formulation at the time was motivated by the Poisson geometric data.  But this time, 
to make things to fit the algebraic framework given in \cite{VV}, let us work with the 
pair $(G_1,H/Z)$, where $H/Z$ is the dual of $G_2$.  To be more precise, consider: 
$$
G_1=\bigl\{(0,0,r):r\in\mathbb{R}\bigr\}\qquad
{\text {and }}\qquad H/Z=\bigl\{(x,y,0):x,y\in\mathbb{R}^n\bigr\}.
$$
We may use the (partial) Fourier transform to move between the functions on $G_2$ and those 
on $H/Z$.  It is not difficult to see that $(G_1,H/Z)$ forms a matched pair.  By abuse of 
notation, we again denote the actions by  $\alpha:G_1\times H/Z\to H/Z$ and $\gamma:H/Z\times G_1
\to G_1$.  We then have:
$$
\alpha_r(x,y):=\left(x+r\sum_{i,k=1}^nJ_{ik}y_k\mathbf{x}_i,y\right),
\qquad\gamma_{(x,y)}(r):=r.
$$
At the algebra level, we obtain the ${}^*$-isomorphism $\tau:L^{\infty}(G_1)\otimes L^{\infty}(H/Z)
\to L^{\infty}(G_1)\otimes L^{\infty}(H/Z)$, given by
$$
\bigl(\tau(f)\bigr)\bigl(r;(x,y)\bigr)=f\bigl(\gamma_{(x,y)}(r);\alpha_r(x,y)\bigr)
=f\left(r;x+r\sum_{i,k=1}^nJ_{ik}y_k\mathbf{x}_i,y\right).
$$
In fact, these computations were carried out earlier, though implicitly, in our discussion 
following Proposition~\ref{Z} leading up to Proposition~\ref{V_Theta}.  

\begin{lem}\label{matchinglemma}
As above, consider the matched pair $(G_1,H/Z)$, together with the corresponding actions 
$\alpha$ and $\gamma$.  Define ${\mathcal U}:G_1\times G_1\times H/Z\to\mathbb{T}$ and 
${\mathcal V}:G_1\times H/Z\times H/Z\to\mathbb{T}$, given by 
$$
{\mathcal U}\equiv\operatorname{Id},\qquad
{\text {and }}\qquad
{\mathcal V}\bigl(r;(x,y),(x',y')\bigr)=\bar{e}\left[\frac{{r}^2}{2}\sum_{i,k}J_{ik}y'_ky_i\right]
\bar{e}\bigl[r\beta(x',y)\bigr].
$$
Then $(\tau,{\mathcal U},{\mathcal V})$ is a ``cocycle matching'' of $L^{\infty}(G_1)$ and 
$L^{\infty}(H/Z)$, with their natural quantum group structures.
\end{lem}

\begin{rem}
Observe that ${\mathcal V}$ is such that ${\mathcal V}\bigl(r;(x,y),(x',y')\bigr)
=\sigma^r\bigl((x',y'),(x,y)\bigr)$, where $\sigma$ is the cocycle function given 
in Proposition~\ref{V_Theta}\,(2).
\end{rem}

\begin{proof}
Using the definition, we can verify the cocycle conditions given in equation~(4.2) of \cite{VV}. 
Namely, the maps ${\mathcal U}$ and ${\mathcal V}$ satisfy 
\begin{align}
\bullet\ &{\mathcal U}\bigl(r,r';\alpha_{r''}(x,y)\bigr){\mathcal U}\bigl(r+r',r'';(x,y)\bigr)
={\mathcal U}(r',r'';(x,y)\bigr){\mathcal U}\bigl(r,r'+r'';(x,y)\bigr),  \notag \\
\bullet\ &{\mathcal V}\bigl(\gamma_{(x,y)}(r);(x',y'),(x'',y'')\bigr)
{\mathcal V}\bigl(r;(x,y),(x''+x',y''+y')\bigr)  \notag \\
&={\mathcal V}\bigl(r;(x,y),(x',y')\bigr){\mathcal V}\bigl(r;(x'+x,y'+y),(x'',y'')\bigr),  \notag \\
\bullet\ &{\mathcal V}\bigl(r+r';(x,y),(x',y')\bigr)\overline{{\mathcal U}\bigl(r,r';(x'+x,y'+y)\bigr)} 
\notag \\
&=\overline{{\mathcal U}\bigl(r,r';(x,y)\bigr)}\,\overline{{\mathcal U}\bigl(\gamma_{\alpha_{r'}(x,y)}(r),
\gamma_{(x,y)}(r');(x',y')\bigr)}   \notag \\
&\quad\cdot{\mathcal V}\bigl(r;\alpha_{r'}(x,y),\alpha_{\gamma_{(x,y)}(r')}(x',y')\bigr)
{\mathcal V}\bigl(r';(x,y),(x',y')\bigr).
\notag
\end{align}
This is to be expected, considering that ${\mathcal V}$ comes from the cocycle function $\sigma$. 
Thus by Lemma~4.11 of \cite{VV}, we prove the result.
\end{proof}

Therefore by Definition~2.2 and Theorem~2.13 both of \cite{VV}, we obtain the ``cocycle bicrossed 
product'' $M=L^{\infty}(G_1)_{\alpha,{\mathcal U}}\ltimes L^{\infty}(H/Z)$, which is a locally 
compact quantum group.  The associated dual locally compact quantum group is denoted by $\hat{M}$ 
(see again Theorem~2.13 of \cite{VV}).  By construction, it turns out that our $(S,\Delta)$ obtained 
in the previous section is really the $C^*$-algebra counterpart to $\hat{M}$.  The result is below:

\begin{theorem}\label{lcqg}
Our $C^*$-bialgebra $(S,\Delta)$ is none other than the dual of the cocycle bicrossed product 
obtained by the matched pair $(G_1,H/Z)$ and the cocycle maps ${\mathcal U}$ and ${\mathcal V}$. 
Therefore, we conclude that $(S,\Delta)$ is itself a locally compact quantum group.
\end{theorem}

\begin{proof}
An efficient way is to work with the multiplicative unitary operators.  So consider 
$W_1\in{\mathcal B}\bigl(L^2(G_1\times G_1)\bigr)$ and $W_2\in{\mathcal B}
\bigl(L^2(H/Z\times H/Z)\bigr)$ such that
$$
W_1\xi(r;r')=\xi(r;r+r'),\qquad
\hat{W}_2\zeta(x,y;x',y')=\zeta(x,y;x'+x,y'+y),
$$
for $\xi\in L^2(G_1\times G_1)$ and $\zeta\in L^2(H/Z\times H/Z)$.  They determine the 
natural quantum group structures on $L^{\infty}(G_1)$ and $L^{\infty}(H/Z)$.  By Definition~2.2 
of \cite{VV}, as well as the discussion in Section 4.4 of the same paper, the quantum groups 
$M$ and $\hat{M}$ are determined by the multiplicative unitary operator $\hat{W}\in{\mathcal B}
\bigl(L^2(G_1\times H/Z\times G_1\times H/Z)\bigr)$, defined by
$$
\hat{W}=(\gamma\otimes\operatorname{id}\otimes\operatorname{id})\bigl((W_1\otimes1){\mathcal U}^*\bigr)
(\operatorname{id}\otimes\operatorname{id}\otimes\alpha)\bigl({\mathcal V}(1\otimes\hat{W}_2)\bigr).
$$
In our case, it becomes:
\begin{align}
&\hat{W}\xi(r;x,y;r';x',y')  \notag \\
&=\overline{{\mathcal U}\bigl(\gamma_{(x,y)}(r);-\gamma_{(x,y)}(r)+r';(x',y')\bigr)}
{\mathcal V}\bigl(r;(x,y),\alpha_{[-\gamma_{(x,y)}(r)+r']}(x',y')\bigr)   \notag \\
&\quad\cdot\xi\bigl(r;\alpha_{[-\gamma_{(x,y)}(r)+r']}(x',y')+(x,y);-\gamma_{(x,y)}(r)+r';(x',y')\bigr)
\notag \\
&={\mathcal V}\bigl(r;(x,y);\alpha_{[r'-r]}(x',y')\bigr)
\xi\bigl(r;(x,y)+\alpha_{[r'-r]}(x',y');r'-r;x',y'\bigr)  \notag \\
&=\bar{e}\left[\frac{{r}^2}{2}\sum_{i,k}J_{ik}y'_ky_i\right]
\bar{e}\bigl[r\beta\bigl(x'+(r'-r)\sum_{i,k}J_{ik}y'_k\mathbf{x}_i,y\bigr)\bigr]  \notag \\
&\quad\cdot\xi\bigl(r;x+x'+(r'-r)\sum_{i,k}J_{ik}y'_k\mathbf{x}_i,y+y';r'-r,x',y'\bigr)  \notag\\
&=e\left[\frac{{r}^2}{2}\sum_{i,k}J_{ik}y'_ky_i\right]
\bar{e}[rr'\sum_{i,k}J_{ik}y'_ky_i]\bar{e}\bigl[r\beta(x',y)\bigr]   \notag \\
&\quad\cdot\xi\bigl(r;x+x'+(r'-r)\sum_{i,k}J_{ik}y'_k\mathbf{x}_i,y+y';r'-r,x',y'\bigr).
\notag
\end{align}
By general theory, it is known to be multiplicative, so that $\hat{W}\in\hat{M}\otimes M$.
The right slices of $\hat{W}$ generate $\hat{M}$ while the left slices of $\hat{W}$ generate $M$.  

Now consider an involutive operator $K\in{\mathcal B}\bigl(L^2(G_1\times H/Z)\bigr)$, defined by 
\begin{equation}\label{(involutiveK)}
K\xi(r;x,y)=\overline{\xi(-r;x+r\sum_{i,k}J_{ik}y_k\mathbf{x}_i,y)}.
\end{equation}
We will postpone the discussion of the nature of the operator $K$ for the time being (It has 
to do with the ``antipode'' map on our quantum group: See Proposition~\ref{antipodekappa}.). 
Using this, define the operator $\hat{V}$ by
$$
\hat{V}=(K\otimes K)\Sigma\hat{W}^*\Sigma(K\otimes K).
$$
Here $\Sigma$ denotes the flip.  Then $\hat{V}$ is also multiplicative, and the general theory shows that 
$\hat{V}\in M'\otimes\hat{M}$, where $M'$ is the commutant of $M$.  See Proposition~2.15 of \cite{KuVavN}, 
with the understanding that their $J$ operator is $K$ here, so that we do not cause any confusion with 
the skew-symmetric matrix $J$ in our case.  The left slices of $\hat{V}$ generate $\hat{M}$ while the 
right slices of $\hat{V}$ generate $M'$.  

After a straightforward computation using the formulas obtained above, we have: 
\begin{align}
&\hat{V}\xi(r;x,y;r';x',y')   \notag \\
&=e\left[\frac{{r'}^2}{2}\sum_{i,k}J_{ik}y_k(y'_i-y_i)\right]
\bar{e}\bigl[r'\beta(x,y'-y)\bigr]  \notag \\
&\quad\cdot\xi(r+r';x-r'\sum_{i,k}J_{ik}y_k\mathbf{x}_i,y;r';x'-x+r'\sum_{i,k}J_{ik}y_k\mathbf{x}_i,y'-y).
\notag
\end{align}
Compare this result with the definition of the multiplicative unitary operator $V_{\Theta}$ 
we constructed in Propostion~\ref{V_Theta}\,(2), which is exactly the same! [To be really precise, 
we need to flip the $(x,y)$ and the $r$.]  The multiplicative unitary operators being the same 
means that the $C^*$-(or v.N) algebras they generate must agree.  In particular, considering 
the left slices of $\hat{V}=V_{\Theta}$, we conclude that at the $C^*$-algebra level, $\hat{M}$ 
and $S$ must coincide.  It follows that our $(S,\Delta)$ is actually a $C^*$-algebraic locally 
compact quantum group, whose von Neumann algebra envelope is $\hat{M}$.
\end{proof}

\subsection{Other structure maps: Antipode and Haar weight}

While the proof that $(S,\Delta)$ is a quantum group is done, it will be still useful to 
know its other quantum group structure maps, namely, the antipode map and the Haar weight. 
We will try to be brief here (skipping some details), but we wish to point out some nice 
correspondence relations between the classical (Poisson) data and the quantum level, 
strengthening our case that $(S,\Delta)$ is essentially a ``quantized $C_0(G)$''.

Correctly constructing the antipode map from the definitions is rather technical. 
See the main papers \cite{KuVa}, \cite{KuVavN}, and also a new treatment given 
in \cite{VDvN}, which uses the Tomita--Takesaki theory.  For our purposes, though, 
we will just use the following characterization of the antipode, denoted here by 
$\kappa$, given in terms of the multiplicative unitary operator:
\begin{equation}\label{(antipode)}
\kappa\bigl((\omega\otimes\operatorname{id})(V_{\Theta})\bigr)
=(\omega\otimes\operatorname{id})(V_{\Theta}^*).
\end{equation}
The subspace consisting of the elements $(\omega\otimes\operatorname{id})
(V_{\Theta})$, for $\omega\in{\mathcal B}({\mathcal H})_*$, is dense in $S$, 
and forms a core for $\kappa$.

At the level of the dense subspace of functions in $C_c(H/Z\times G_1)$, in the $(x,y;r)$ 
variables, the antipode $\kappa$ in our case takes the following form. 

\begin{prop}\label{antipodekappa}
Let $\kappa:C_c(H/Z\times G_1)\to C_c(H/Z\times G_1)$ be defined by
$$
\bigl(\kappa(f)\bigr)(x,y,r)=\bar{e}\left[\frac{r^2}{2}\sum_{i,k}J_{ik}y_iy_k\right]
\bar{e}\bigl[r\beta(x,y)\bigr]f\left(-x-r\sum_{i,k}J_{ik}y_k\mathbf{x}_i,-y,-r\right).
$$
This map corresponds to the definition of $\kappa$ given in equation~\eqref{(antipode)}, 
and turns out to be a bounded map.  By general theory, its extension to the $C^*$-algebra $S$, 
still denoted by $\kappa$, is the antipode map on $(S,\Delta)$.  Moreover, the antipode 
map $\kappa$ is related with the operator $K$ in equation~\eqref{(involutiveK)} by
$$
K(L_f)^*K=L_{\kappa(f)},\qquad f\in C_c(H/Z\times G_1),
$$
where $L_f\in S$ denotes the operator realization of the function $f$.  It follows that 
$\kappa^2\equiv\operatorname{Id}$.
\end{prop}

\begin{proof}
For $\eta,\zeta\in C_c(H/Z\times G_1)$, consider $\omega_{\eta,\zeta}\in{\mathcal  B}
({\mathcal H})_*$, defined by $\omega_{\eta,\zeta}(T):=\langle T\eta,\zeta\rangle$. 
Since $C_c(H/Z\times G_1)$ is dense in ${\mathcal H}$, it is clear that the 
$\omega_{\eta,\zeta}$ are dense in ${\mathcal  B}({\mathcal H})_*$.  Meanwhile, by 
a straightforward calculation, we can show that the operator $(\omega_{\eta,\zeta}
\otimes\operatorname{id})(V_{\Theta})$ can be realized as $L_f$, where $f$ is a 
function contained in $C_c(H/Z\times G_1)$ defined by
$$
f(x,y;r)=\int\eta(x,y,r+\tilde{r})\overline{\zeta\left(x+r\sum_{i,k}J_{ik}y_k\mathbf{x}_i,y,
\tilde{r}\right)}\,d\tilde{r}.
$$
Similarly, $(\omega_{\eta,\zeta}\otimes\operatorname{id})(V_{\Theta}^*)$ can be 
realized as $L_g$, where
\begin{align}
g(x,y;r)&=\int\bar{e}\left[\frac{r^2}{2}\sum_{i,k}J_{ik}y_ky_i\right]\bar{e}\bigl[r\beta(x,y)\bigr]
\notag \\
&\qquad\eta\left(-x-r\sum_{i,k}J_{ik}y_k\mathbf{x}_i,-y,-r+\tilde{r}\right)
\overline{\zeta(-x,-y,\tilde{r})}\,d\tilde{r}.
\notag
\end{align}
By equation~\eqref{(antipode)}, the function $g$ is none other than $\kappa(f)$. 
Comparing it with the expression for $f$ above, we obtain the result of the proposition.  
Since the $\omega_{\eta,\zeta}$ are dense in ${\mathcal  B}({\mathcal H})_*$, this 
characterization of the $\kappa$ map is sufficient.

Meanwhile, we also have: $K(L_f)^*K=L_{\kappa(f)}$, where $f\in C_c(H/Z\times G_1)$ 
and $\kappa(f)$ is as above.  Calculation is straightforward.  Since $K$ is a bounded 
operator and involutive, this implies that $\kappa:L_f\mapsto L_{\kappa(f)}$ can be extended 
to a bounded map on all of the $C^*$-algebra $S$, with $\kappa^2\equiv\operatorname{Id}$.
\end{proof}

\begin{rem}
Note that when $\sigma\equiv1$, we have:
$$
\bigl(\kappa(f)\bigr)(x,y,r)=f\left(-x-r\sum_{i,k}J_{ik}y_k\mathbf{x}_i,-y,-r\right).
$$
If we express this in the $(p,q,r)$-variables, by the partial Fourier transform, it becomes: 
$$
\bigl(\kappa(f)\bigr)(p,q,r)=f\left(-p,-q+r\sum_{i,k=1}^n J_{ik}p_i\mathbf{q}_k,-r\right)
=f\bigl((p,q,r)^{-1}\bigr).
$$
What all this means is that in the commutative case (when $\sigma\equiv1$), the antipode map 
is just taking the inverse in the group $G$.  This again strengthens our point that 
$(S,\Delta)$ is a ``quantized $C_0(G)$''.
\end{rem}

Finally, let us turn our attention to the Haar weight on our quantum group $(S,\Delta)$. 
At the classical level, recall that the group structure on $G$ was chosen in 
equation~\eqref{(groupmult)} so that an ordinary Lebesgue measure on $G=\mathbb{R}^{2n+1}$ 
becomes its (left invariant) Haar measure.  This suggests us to build the Haar weight on 
$(S,\Delta)$ from the Lebesgue measure on $G$.  At the level of the functions in ${\mathcal A}
={\mathcal S}_{3c}(G)$, this suggestion is manifested in Definition~\ref{haar} below:

\begin{defn}\label{haar}
(1). On ${\mathcal A}$, define a linear functional $\varphi$ by
$$
\varphi(f)=\int f(p,q,r)\,dpdqdr.
$$

(2). At the level of the functions in ${\mathcal S}_{3c}(H/Z\times G_1)$, in the $(x,y;r)$ variables, 
this is equivalent to the linear functional $\varphi_S$ below:
$$
\varphi_S(f)=\int f(0,0;r)\,dr.
$$
\end{defn}

\begin{lem}\label{haarlemma}
Let $\varphi_S$ be the linear functional given in Definition~\ref{haar}.  It satisfies the following 
``left invariance property'':
\begin{equation}\label{(leftinvariance)}
(\operatorname{id}\otimes\varphi_S)\bigl((1\otimes f)(\Delta g)\bigr)
=\kappa\bigl((\operatorname{id}\otimes\varphi_S)((\Delta f)(1\otimes g))\bigr),
\end{equation}
for $f,g\in{\mathcal S}_{3c}(H/Z\times G_1)$.
\end{lem}

\begin{proof}
Here, the expression $(1\otimes f)(\Delta g)$ means the function $F$ such that 
$(L\otimes L)_F=(1\otimes L_f)\bigl(\Delta(L_g)\bigr)$.  By using the definitions 
of the comultiplication, the product on $S$, and the definition of the functional 
$\varphi_S$, we have:
\begin{align}
&(\operatorname{id}\otimes\varphi_S)\bigl((1\otimes f)(\Delta g)\bigr)(x,y;r)  \notag \\
&=\int f(-x+\tilde{r}\sum_{i,k}J_{ik}y_k\mathbf{x}_i,-y,\tilde{r})
g(x-\tilde{r}\sum_{i,k}J_{ik}y_k\mathbf{x}_i,y,r+\tilde{r})  \notag \\
&\qquad\cdot e\bigl[\tilde{r}\beta(x,y)\bigr]
\bar{e}\left[\frac{\tilde{r}^2}{2}\sum_{i,k}J_{ik}y_iy_k\right]\,d\tilde{r}.
\notag
\end{align}
By a similar computation, we also have:
\begin{align}
&(\operatorname{id}\otimes\varphi_S)\bigl((\Delta f)(1\otimes f)\bigr)(x,y;r)  \notag \\
&=\int f(x-\tilde{r}\sum_{i,k}J_{ik}y_k\mathbf{x}_i,y,r+\tilde{r})
g(-x+\tilde{r}\sum_{i,k}J_{ik}y_k\mathbf{x}_i,-y,\tilde{r})  \notag \\
&\qquad\cdot e\bigl[\tilde{r}\beta(x,y)\bigr]
\bar{e}\left[\frac{\tilde{r}^2}{2}\sum_{i,k}J_{ik}y_iy_k\right]\,d\tilde{r}.
\notag
\end{align}
Therefore, by using the definition of the antipode map $\kappa$, as obtained in 
Proposition~\ref{antipodekappa}, we can show the following:
$$
\kappa\bigl((\operatorname{id}\otimes\varphi_S)((\Delta f)(1\otimes g))\bigr)(x,y;r)
=(\operatorname{id}\otimes\varphi_S)\bigl((1\otimes f)(\Delta g)\bigr)(x,y;r).
$$
\end{proof}

In Kac algebra theory, equation~\eqref{(leftinvariance)} has been used to define the 
left invariance of the Haar weight.  Our proof was given only at the function level, 
but it nevertheless provides some justification to our choice of $\varphi_S$.

In general, jumping up from the linear functional at the level of the functions to the 
weight at the operator level can be quite technical.  See papers on the Haar weights on 
general locally compact quantum groups, like \cite{VDHaar}, \cite{VDoamp}.  While we can 
actually proceed using a similar approach as in \cite{BJKppha}, we made a decision above 
to take advantage of the fact that our example is a case of a cocycle bicrossed product. 
The precise construction of the Haar weight can be found in Propostion~2.9 of \cite{VV} 
(see the remark below).

\begin{rem}
[Some technical remarks.  See \cite{VV}.]
We noted in Theorem~\ref{lcqg} that $(S,\Delta)$ is the $C^*$-algebra counterpart 
to the locally compact quantum group $\hat{M}$, obtained from the matched pair 
$\bigl(L^{\infty}(G_1),L^{\infty}(H/Z)\bigr)$ and the cocycle maps ${\mathcal U}$ 
and ${\mathcal V}$.  [To be really precise, we need to flip the $(x,y)$ and the 
$r$.]  Recall also the actions $\alpha$ and $\gamma$ from Section~4.1.
By general theory on the cocycle bicrossed products, $\hat{M}$ is generated 
by $\gamma(M_1)$ and $\bigl\{(\operatorname{id}\otimes\operatorname{id}
\otimes\omega)({\mathcal V}(1\otimes\hat{W}_2)):\omega\in (M_2)_*\bigr\}$, 
where $M_1=L^{\infty}(G_1)$ and $M_2=L^{\infty}(H/Z)$, in our case.  It turns out 
that $\gamma(M_1)$ is the fixed point algebra of the dual action, $\hat{\gamma}$, 
of $(\hat{M}_2,\hat{\Delta}_2^{\operatorname{cop}})$ on $\hat{M}$.  Then 
$T=(\operatorname{id}\otimes\operatorname{id}\otimes\hat{\varphi}_2)\hat{\gamma}$, 
where $\hat{\varphi}_2$ is the Haar weight on $\hat{M}_2$, defines a normal, 
faithful operator-valued weight from $\hat{M}$ to $\gamma(M_1)$.  From the Haar 
weight $\varphi_1$ on $M_1$, we can then define the normal, semi-finite weight 
$\varphi_{\hat{M}}$ on $\hat{M}$, by $\varphi_{\hat{M}}=\varphi_1\circ\gamma^{-1}
\circ T$.  From $\varphi_{\hat{M}}$, by restriction and the flip: $(r;x,y)\mapsto(x,y;r)$, 
we would obtain the Haar weight on $S$.
\end{rem}

Having given these remarks and Lemma~\ref{haarlemma}, together with the knowledge that 
Haar weight is unique (up to a scalar multiplication), we will accept that the linear 
functional $\varphi_S$ above does indeed extend to the correct Haar weight on $(S,\Delta)$.

\bigskip

\providecommand{\bysame}{\leavevmode\hbox to3em{\hrulefill}\thinspace}
\providecommand{\MR}{\relax\ifhmode\unskip\space\fi MR }
\providecommand{\MRhref}[2]{%
  \href{http://www.ams.org/mathscinet-getitem?mr=#1}{#2}
}
\providecommand{\href}[2]{#2}

\end{document}